\documentclass[reqno,letterpaper]{gthack}   
\usepackage[margin=1.3in, left=1.5in, right=1.3in]{geometry}

\usepackage[numbers]{natbib}
\setcitestyle{open={[},close={]}}

\usepackage{xfrac}

\titleformat{\section}
  {\normalfont\fontsize{14}{5}\bfseries}{\thesection}{1em}{}

\usepackage{graphicx}
\usepackage{enumitem}
\usepackage{tikz-cd}

\usepackage{comment}
\usepackage[margin=.5cm,font=small,labelfont=small]{caption}

\title{Exotically knotted disks and complex curves}
\author{Kyle Hayden} \address{Columbia University, New York, NY 10027} 
\email{hayden@math.columbia.edu}

\theoremstyle{plain}
\newtheorem{thm}{Theorem}[section]   \newtheorem{lem}[thm]{Lemma}
              \newtheorem{prop}[thm]{Proposition}

\newtheorem{ques}[thm]{Question}
\newtheorem*{ques*}{Question}

\newtheorem*{thmcusp*}{Theorem B}

\newtheorem{mainthm}{Theorem}[section]

\theoremstyle{definition}

\newtheorem{rem}[thm]{Remark}       \newtheorem*{rem*}{Remark}             
\theoremstyle{remark}
\newtheorem{ex-main}[thm]{Example}

\newcommand{\cc}{\mathbb{C}}
\newcommand{\rr}{\mathbb{R}}
\newcommand{\zz}{\mathbb{Z}}
\newcommand{\st}{{\mathrm{st}}}

\renewcommand{\S}{\textsection}

\newcommand{\cee}{\mathbb{C}}

\newcommand{\cp}{\cee {P}^2}

\newcommand{\surf}{F}
\newcommand{\surfb}{S}

\begin{document}

\begin{abstract}
\vspace{-.1in}
This paper studies properly embedded surfaces in the 4-ball that are exotically knotted (i.e., topologically but not smoothly isotopic), and leverages this local  phenomenon to study surfaces in larger 4-manifolds. The main results provide a new construction of exotically knotted surfaces, including exotic slice surfaces of all genera in the 4-ball and exotic closed surfaces of large negative self-intersection in larger 4-manifolds.  The construction is well-suited to the complex and symplectic settings, providing the first examples of exotically knotted complex curves and  symplectic 2-spheres. Along the way, we articulate some diagrammatic tools for constructing symplectic surfaces and complex curves.  

\vspace{-.043in}

\quad We also use local knotting to investigate the geography problem for knot groups, constructing the first examples of exotically knotted surfaces in closed, simply connected 4-manifolds whose knot groups contain nonabelian free subgroups, hence are not expected to be ``good'' groups in the sense of surgery theory. 
\end{abstract}

\maketitle

\titleformat{\section}{\large\bfseries}{}{0pt}{\center \thesection.  }

\titlespacing{\section}{0pt}{*4}{*1.5}

\vspace{-.6in}

\section{Introduction}

A pair of smooth surfaces in a smooth 4-manifold $X$ are said to be \emph{exotic} (or \emph{exotically knotted}) if the surfaces are isotopic through homeomorphisms of $X$ but not through diffeomorphisms of $X$.  Much of the progress in studying knotted surfaces has been sparked by questions from algebraic geometry and symplectic topology. For example, if $X$ is a closed, simply connected K\"ahler surface, then any two smooth, closed complex curves that are homologous are in fact smoothly isotopic \cite[\S1]{fs:symplectic}. But if the surfaces are only required to be smooth, then there are many examples of infinite families of surfaces that are pairwise topologically but not smoothly isotopic. 
 The majority of such examples are surfaces of positive genus obtained using the ``rim surgery'' technique pioneered by Fintushel-Stern \cite{fs:surfaces} and further developed by several authors, though a handful of other constructions have arisen \cite{akbulut:2-spheres,akmr:stable,baykur-hayano,schwartz}.

This paper introduces a construction of exotically knotted surfaces that  pushes forward the state of the art  on three fronts: the minimal topology of the surfaces (including disks and spheres), the simplicity of the ambient 4-manifolds (including $B^4$), and the rigid geometry of the surfaces.  We offer three perspectives on the construction: (i) using handle diagrams, (ii) using movies of link diagrams, and (iii) using braid factorizations. 

We discuss the main topological results in \S1.1, followed by a more detailed discussion of the symplectic and complex settings in \S1.2-1.3. For additional  context, we outline our approach and compare it with previous constructions of exotic smooth and symplectic surfaces at the end of this introduction in \S1.4.

\titleformat{\subsection}[runin]{\bfseries}{}{0pt}{\thesubsection \ \  }

\subsection{Main topological results.} 

Our most fundamental result is motivated by the open question of whether there exist closed, orientable surfaces in $S^4$ that are exotically knotted. In the relative setting of properly embedded surfaces in $B^4$, we prove:

 \begin{mainthm}\label{thm:holomorphic}
For any integer $g \geq 0$, there exist infinitely many knots in $S^3$ that each bound a pair of  properly embedded, smooth, orientable surfaces of genus $g$ in $B^4$ that are isotopic through homeomorphisms  but not diffeomorphisms of $B^4$.  Moreover, these surfaces may be embedded holomorphically in $B^4 \subset \cc^2$.
 \end{mainthm}

 To the author's knowledge, Theorem~\ref{thm:holomorphic} provides the first known examples of smoothly exotic complex curves in any 4-manifold.  The construction is explicit enough that many examples can be drawn by hand.  These surfaces may also be capped off inside of larger 4-manifolds to produce exotically knotted closed surfaces, and the construction is easily extended to produce arbitrarily large finite collections of exotic surfaces.

\begin{ex-main} The ribbon-immersed disks in Figure~\ref{fig:explicit} give rise to exotic disks in $B^4$.\end{ex-main}

\begin{figure}
\center
\includegraphics[width=.95\linewidth]{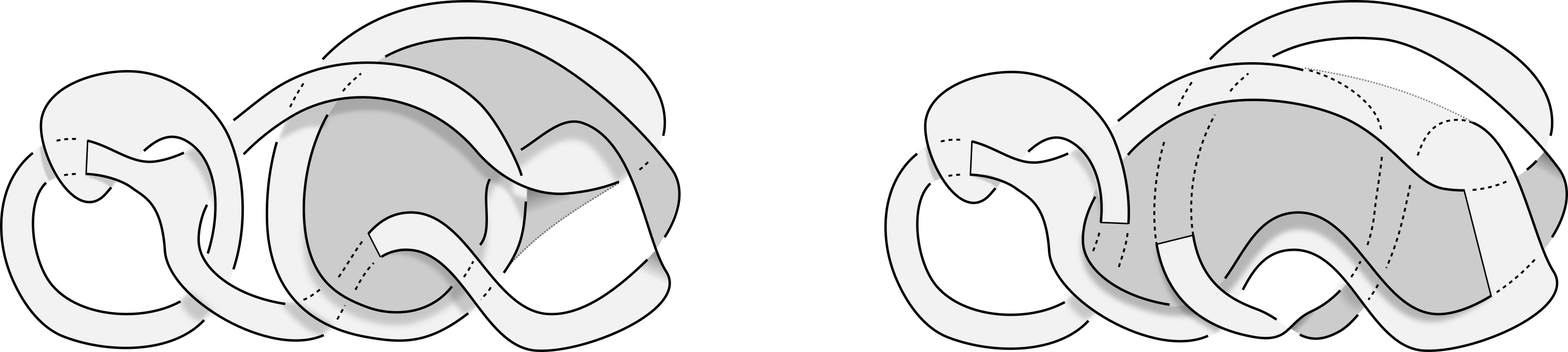}
\caption{An exotic pair of ribbon disks bounded by the same knot.}\label{fig:explicit}

\vspace{.1in}

\end{figure}

Constructions of exotically knotted surfaces typically constrain the fundamental group of the surface complement, called the \emph{knot group} of the surface. The knot group is usually required to be a ``good'' group  in the sense of surgery theory \cite{freedman-quinn,freedman-teichner:i,freedman-teichner:ii}, which ensures that the topological $s$-cobordism theorem holds.  This class  includes finite groups and $\zz$, whereas free groups of rank $\geq 2$ are \emph{not} expected to be good.  A wealth of nontrivial ``good'' groups have been realized as knot groups of exotic surfaces \cite{kim:twist,kim-ruberman:non-simply,kim-ruberman:trivial,hoffman-sunukjian,conway-powell:Z,torres}.  In contrast, we leverage  knotted surfaces in $B^4$  to produce exotic surfaces whose knot groups  are \emph{not} expected to be ``good''.

To state the result, we consider smooth  surfaces $\surf$ that are obtained by resolving \emph{cusps} of singular symplectic surfaces. These are singularities locally modeled on a complex curve $z^p + w^q =0$ in $B^4 \subset \cc^2$ for coprime integers $p,q \geq 2$, i.e., the cone on a $(p,q)$-torus knot. To resolve the surface, the singularity can be replaced with the standard smooth surface in $B^4$ bounded by the $(p,q)$-torus knot. (See \S\ref{subsec:cuspidal} for precise definitions.) Such singular surfaces are abundant in many symplectic 4-manifolds, and their resolutions necessarily intersect a 4-ball inside $X$ in a well-controlled subsurface of positive genus. 


\begin{mainthm}\label{thm:cuspidal}
Suppose  $\surf$ is a smooth surface obtained by resolving a cusp of a singular symplectic curve in a closed or convex symplectic 4-manifold $X$.\begin{enumerate}
\item[\normalfont \text{(a)}] If $\surf \cdot \surf \geq 0$, then there are infinitely many smoothly embedded surfaces in $X$ that are topologically but not smoothly isotopic to $\surf$.
\item [\normalfont \text{(b)}] Moreover, if $g(\surf)>1$ and the meridian to $\surf$ is a primitive element of finite order $d>1$ in $H_1(X \! \setminus \! \surf)$, then the  class $[\surf]  \in H_2(X)$ contains infinitely many pairwise exotic surfaces whose knot groups contain nonabelian free subgroups.
\end{enumerate}
\end{mainthm}

\smallskip

\begin{rem*}
Unlike the other main results, this theorem is obtained using rim surgery. It appears that similar phenomena can be produced by modifying our main construction, but the results do not apply as broadly.
\end{rem*}

\begin{ex-main} Each homology class of degree $d>3$ in $\cp$ is represented by a smooth surface $F$ of genus $g(F)>1$ with $H_1(\cp\setminus F)\cong \zz/d$,  obtained by resolving the cusps of a singular complex curve.  Hence these homology classes also contain exotic surfaces whose knot groups admit nonabelian free subgroups.
\end{ex-main}

Localized knotting also allows for the construction of exotic singular surfaces. While our main interest in such phenomena arises from the study of singular complex curves, we highlight an amusing corollary of Theorem~\ref{thm:holomorphic}: the existence of exotically knotted surfaces in $S^4$ in the \emph{piecewise-linear} category.

\begin{prop}\label{prop:PL}
For any integer $g \geq 0$, there exist  pairs of piecewise-linearly embedded, closed, orientable surfaces of genus $g$ in $S^4$ that are isotopic through (topological) homeomorphisms of $S^4$ yet not through piecewise-linear homeomorphisms of $S^4$.
\end{prop}



\vspace{-.15in}

\subsection{The symplectic setting.} Symplectic surfaces offer a compromise between the rigidity of closed complex curves and the flexibility of arbitrary smooth surfaces. In $\cp$, for example, it is conjectured that any two smoothly embedded symplectic surfaces in the same homology class are symplectically isotopic; to date, this has been proven for homology classes of degree at most 17 \cite{siebert-tian}. And in \emph{any} closed symplectic 4-manifold, a symplectic 2-sphere with self-intersection greater than $-2$ is unique up to isotopy in its homology class (cf \cite[Proposition~3.2]{li-wu}). 

However, the principle fails in general. Works of several authors, including Fintushel-Stern \cite{fs:symplectic} and Park-Poddar-Vidussi \cite{park-poddar-vidussi}, provide an array of symplectic 4-manifolds containing symplectic surfaces of positive genus that are homologous but non-isotopic; see \cite{etgu:survey} for a concise survey.  While many of these surfaces are distinguished by the fundamental groups of their complements, others require more sensitive tools from gauge theory.  Moreover, a subset of these surfaces can be constructed to have simply connected complements, implying that  they are  topologically isotopic \cite{freedman,boyer,sunukjian} and thus smoothly exotic.

 By extending the construction underlying Theorem~\ref{thm:holomorphic}, we produce a new construction of smoothly exotic symplectic surfaces, including the first such examples of 2-spheres and surfaces with nonzero self-intersection (cf~\cite[\S8]{etgu:survey}).

\begin{mainthm}\label{thm:symp}
For any integer $g \geq 0$, there exist infinitely many  convex symplectic 4-manifolds each of which contains a pair of closed symplectic surfaces of genus $g$  that are topologically isotopic yet are not equivalent under any ambient diffeomorphism.
\end{mainthm}

These also give the first examples of exotically knotted spheres in 4-manifolds with nonvanishing gauge-theoretic invariants; see \S\ref{subsec:compare} below for more context. It is straightforward to symplectically embed the 4-manifolds from Theorem~\ref{thm:symp} into \emph{closed} symplectic 4-manifolds. However, the main obstructive tools used in this paper (namely the adjunction inequality) are too coarse to distinguish the surfaces in this context. Though we expect that these symplectic surfaces typically remain exotic in closed symplectic 4-manifolds, we leave this problem open for future investigation.

\vspace{-.1in}

\subsection{The complex setting.}  The exotic complex curves  from Theorem~\ref{thm:holomorphic} can be used to construct  exotic complex curves in other 4-manifolds.  
For example, while there are no smoothly exotic complex curves in any \emph{closed}, simply connected K\"ahler surface, we can find non-compact examples in \emph{open} K\"ahler surfaces.

 \begin{mainthm}\label{thm:open}
There exist open, simply connected complex surfaces (that are Stein and therefore K\"ahler) containing exotically knotted, properly embedded complex curves. \end{mainthm}

We note that working in \emph{open} complex surfaces is a subtler affair than working in compact complex domains with boundary.  Indeed, the proof of Theorem~\ref{thm:holomorphic} does \emph{not} readily imply that interiors of the complex curves are smoothly exotic in the open 4-ball in $\cc^2$. In particular, the following fundamental question remains open.

 \begin{ques}Does $\cc^2$ contain smoothly exotic, properly embedded complex curves?
 \end{ques}

Shifting perspective, exotic surfaces can also be used to construct exotic 4-manifolds. Taking branched covers of $B^4$ over the holomorphic disks from Theorem~\ref{thm:holomorphic}, we obtain: 

\begin{mainthm}\label{thm:stein}
There exist contractible Stein domains with the same contact boundary that are homeomorphic but not diffeomorphic.
\end{mainthm}

To the author's knowledge, all previous examples of exotic contractible Stein domains are distinguished by their contact boundaries; among other things, this implies that they cannot be used for cut-and-paste constructions in the symplectic setting. Relaxing the condition on the algebraic topology of the filling, Akhmedov-Etnyre-Mark-Smith \cite{aems} produced the first examples of Stein fillings of a fixed contact 3-manifold that are homeomorphic but not diffeomorphic; see also \cite{ao:singularity}. While these examples have large second Betti numbers,   Akbulut-Yasui successfully produced smaller examples with $b_2=2$ \cite{akbulut-yasui:small-stein}. However, all of these examples rely on log transforms or the Fintushel-Stern knot surgery operation, hence generally require $b_2>1$. The novel approach using branched covers of complex disks avoids these restrictions on the algebraic topology of the  fillings.

\vspace{-.125in}
\subsection{Comparison with earlier work.}\label{subsec:compare} To date, all earlier methods of constructing smoothly exotic surfaces $\surf,\surf' \subset X$ suffer from at least two of the following three limitations: (1) the surfaces have positive genus, (2) the topological isotopy is only ensured when $X \setminus \surf$ and $X \setminus \surf'$ have finite cyclic fundamental group, and (3) the surfaces $\surf$  and $\surf'$ are not \emph{both} symplectic.

 Our underlying topological construction of pairs of exotic disks $D,D' \subset B^4$ is based on a variant of cork twisting, strengthening and systemizing an example of a related phenomenon found by Akbulut \cite{akbulut:zeeman}.  In contrast with rim surgery, the explicit nature of this construction makes it simple to describe the resulting surfaces using a sequence of link diagrams. Under suitable conditions, we can apply results from \cite{hayden:stein}  to realize the resulting disks $D$ and $D'$ symplectically in $(B^4,\omega_\st)$. We then obtain closed symplectic surfaces $\surf, \surf'$ by capping off these disks $D,D' \subset B^4$ inside a larger symplectic 4-manifold $X$.  The  complements $X \setminus \surf$ and $X\setminus \surf'$ are then distinguished using an adjunction inequality for embedded surfaces in Stein domains \cite{lisca-matic} and an analysis of the branched covers of $X$ along $\surf$ and $\surf'$.

In contrast, the earlier strategy for constructing and detecting smoothly exotic 2-spheres in \cite{akbulut:2-spheres,akmr:stable} is not well-suited to the symplectic setting, nor to the study of surfaces of positive genus. The underlying strategy, also applied  in \cite{ruberman:obstruction}, is based on the existence of exotic 4-manifolds $X$ and $X'$ whose blowups $X \# \cp$ and $X' \# \cp$ are diffeomorphic. If $\surf$ and $\surf'$ denote the 2-spheres given by $\mathbb{C}P^1$ inside the $\cp$-summand of $X\# \cp$ and $X' \# \cp$, respectively, then there is a homeomorphism from $X \#\cp$ to $X' \# \cp$ carrying $\surf$ to $\surf'$. However, any diffeomorphism  carrying $\surf$ to $\surf'$ would descend to a diffeomorphism of $X$ and $X'$ after blowing down the 2-spheres $\surf$ and $\surf'$ (which each have self-intersection $+1$), a contradiction.

In the symplectic setting, however, there can be no smoothly exotic symplectic 2-spheres of self-intersection $\pm 1$. Indeed, as mentioned above, any symplectic 2-sphere of self-intersection greater than $-2$ in a closed symplectic 4-manifold is unique up to isotopy in its homology class \cite[Proposition~3.2]{li-wu}. Moreover, it was pointed out to the author by T.-J. Li \cite{li:email} that this result can be extended to symplectic 4-manifolds with convex boundary using the tools surveyed in \cite[\S9]{wendl:book}. In theory, it remains possible that a pair of symplectic 2-spheres in a 4-manifold $Z$ could be distinguished by a \emph{rational} blowdown, which consists of replacing a tubular neighborhood of the 2-sphere with a rational homology ball that has the same boundary \cite{fs:blowdown}. Such a rational homology ball exists precisely when the 2-sphere has self-intersection $\pm1$ or $\pm 4$. While the case of self-intersection $-4$ is not ruled out above, the manner in which the gauge-theoretic invariants of $Z$ determine those of its rational blowdowns makes it difficult to distinguish the resulting 4-manifolds.

\begin{rem}
A notable  paper on exotic surfaces by Juh\'asz, Miller, and Zemke \cite{jmz:exotic}  appeared while the later sections of this paper were in preparation. Using twisted rim surgery, they produce infinite collections of smooth  surfaces of positive genus in $B^4$ that are topologically isotopic rel boundary yet are not ambiently diffeomorphic.  These surfaces are distinguished using the maps they  induce on knot Floer homology. Our approaches differ in fundamental ways that warrant further investigation. In particular, it can be shown that the  obstruction in \cite{jmz:exotic} cannot distinguish pairs of exotically knotted complex curves, such as those arising in Theorem~\ref{thm:holomorphic}.
\end{rem}

\emph{Organization.} We begin in \S\ref{sec:simple} by explaining the topological construction underlying our main results, giving examples of exotically knotted slice disks in $B^4$. In \S\ref{sec:symp}, we adapt the main construction to the symplectic setting and prove a symplectic version of Theorem~\ref{thm:holomorphic} from which most of the main results are derived. In \S\ref{sec:complex}, we shift to the complex setting, deriving Theorems~\ref{thm:holomorphic}, \ref{thm:open}, and \ref{thm:stein}. Finally, we return to the smooth setting in \S\ref{sec:closed}, studying surfaces in closed 4-manifolds and proving Theorem~\ref{thm:cuspidal}.

\smallskip

\emph{Acknowledgements.} I wish to thank Lisa Piccirillo for several  key conversations,  and Patrick Orson for sharing  technical insights, including the Alexander trick. Thanks also to Inanc Baykur, Anthony Conway, John Etnyre, Dave Gay, Bob Gompf, Tian-Jun Li, and Danny Ruberman for helpful exchanges, and to Frank Swenton for developing and maintaining the Kirby Calculator \cite{klo}, which greatly simplified the experimental phase of this project. This work was supported by NSF grant DMS-1803584.

\vspace{-.35in}

\section{The topological construction}\label{sec:simple}

In this section, we outline the topological construction underlying our pairs of smoothly exotic surfaces. For  variety, we distinguish the examples in this section using an obstruction from knot Floer homology; the proofs of Theorems~\ref{thm:holomorphic}-\ref{thm:symp} given later on use branched covers and  the adjunction inequality for surfaces in Stein domains.

\begin{figure}[b]\center
\smallskip
\def\svgwidth{.925\linewidth}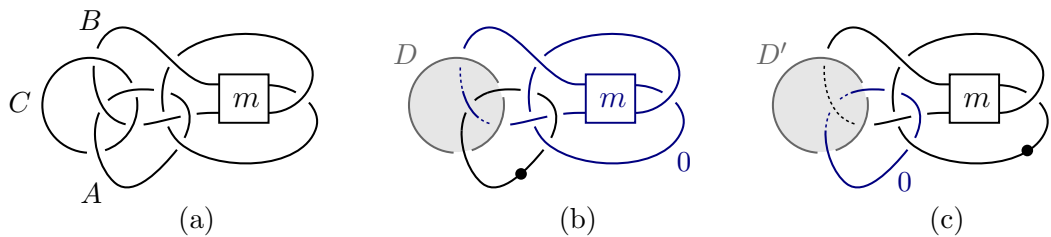
\caption{Disks $D,D' \subset B^4$ associated to the link of unknots $A \cup B \cup C$.}\label{fig:first-link}
\end{figure}

 To begin, let $L$ be a three-component link of unknots $A$, $B$, and $C$ such that $A \cup B$ is a Hopf link and each of the links $A \cup C$ and $B \cup C$ is a two-component unlink. For a running example, see Figure~\ref{fig:first-link}(a); in this and all other figures, a boxed integer indicates positive full twists. We may view $C$ as a knot $K \subset S^3$ using the non-standard surgery description of $S^3$ given by zero-surgery on both $A$ and $B$. Moreover, by viewing $A$ as a dotted circle (i.e.,~carving out a neighborhood of a standard slice disk for $A$ from $B^4$ to yield $S^1 \times B^3$ \cite{GompfStipsicz4}) and attaching a zero-framed 2-handle along $B$, we obtain a non-standard handle diagram for $B^4$; see Figure~\ref{fig:first-link}(b).

Observe that, since $K$ is unknotted in this diagram and does not run over the 1-handle, it naturally bounds a slice disk in $B^4$. More precisely, the  link component $C$ underlying $K$ is split from $A$ (when ignoring $B$), so $C$ bounds an embedded disk in $S^3 \setminus A$ that is unique up to isotopy. The interior of this disk can be pushed into the interior of $S^1 \times B^3$, i.e.,~the exterior of the slice disk for $A$. The resulting disk is disjoint from the 2-handle attached along $B$, giving rise to the desired disk $D \subset B^4$ bounded by $K \subset S^3$. By hypothesis, we may reverse the roles of $A$ and $B$ in this construction, yielding another disk $D'$ in $B^4$ bounded by $K$; see Figure~\ref{fig:first-link}(c).

To draw the resulting knots in the standard diagram for $S^3$, we slide $C$ over the 2-handle until the 1- and 2-handles can be canceled. For example, after isotopy, the knots determined by Figure~\ref{fig:first-link} are drawn in Figure~\ref{fig:simple-knot}(a-b).

\begin{figure}\center
\def\svgwidth{.8\linewidth}
\begingroup%
  \makeatletter%
  \providecommand\color[2][]{%
    \errmessage{(Inkscape) Color is used for the text in Inkscape, but the package 'color.sty' is not loaded}%
    \renewcommand\color[2][]{}%
  }%
  \providecommand\transparent[1]{%
    \errmessage{(Inkscape) Transparency is used (non-zero) for the text in Inkscape, but the package 'transparent.sty' is not loaded}%
    \renewcommand\transparent[1]{}%
  }%
  \providecommand\rotatebox[2]{#2}%
  \newcommand*\fsize{\dimexpr\f@size pt\relax}%
  \newcommand*\lineheight[1]{\fontsize{\fsize}{#1\fsize}\selectfont}%
  \ifx\svgwidth\undefined%
    \setlength{\unitlength}{907.45798667bp}%
    \ifx\svgscale\undefined%
      \relax%
    \else%
      \setlength{\unitlength}{\unitlength * \real{\svgscale}}%
    \fi%
  \else%
    \setlength{\unitlength}{\svgwidth}%
  \fi%
  \global\let\svgwidth\undefined%
  \global\let\svgscale\undefined%
  \makeatother%
  \begin{picture}(1,0.31355333)%
    \lineheight{1}%
    \setlength\tabcolsep{0pt}%
    \put(0,0){\includegraphics[width=\unitlength,page=1]{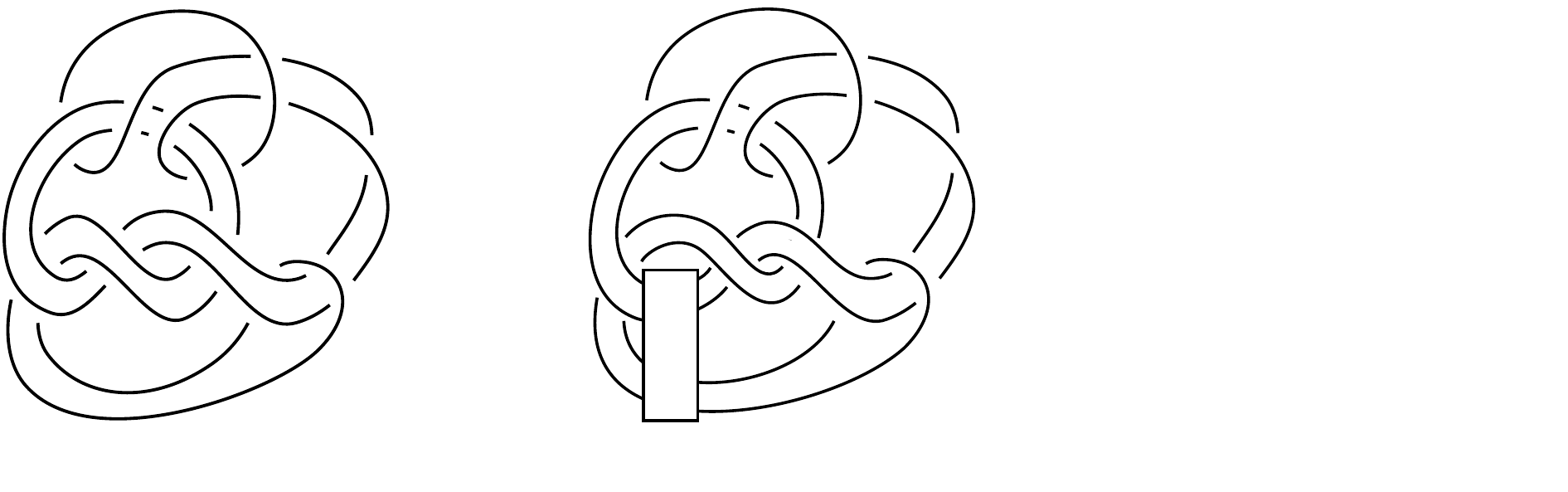}}%
    \put(-0.03483356,0.22085556){\makebox(0,0)[lt]{\lineheight{1.25}\smash{\begin{tabular}[t]{l}$K_0$\end{tabular}}}}%
    \put(0,0){\includegraphics[width=\unitlength,page=2]{simpledisk-fam.pdf}}%
    \put(0.41750276,0.0863734){\makebox(0,0)[lt]{\lineheight{1.25}\smash{\begin{tabular}[t]{l}\text{\footnotesize$m$}\end{tabular}}}}%
    \put(0.32832046,0.21754957){\makebox(0,0)[lt]{\lineheight{1.25}\smash{\begin{tabular}[t]{l}$K_m$\end{tabular}}}}%
    \put(0,0){\includegraphics[width=\unitlength,page=3]{simpledisk-fam.pdf}}%
    \put(0.76918716,0.02745796){\makebox(0,0)[lt]{\lineheight{1.25}\smash{\begin{tabular}[t]{l}$\ell$\end{tabular}}}}%
    \put(0,0){\includegraphics[width=\unitlength,page=4]{simpledisk-fam.pdf}}%
    \put(0.10005794,0.00334464){\makebox(0,0)[lt]{\lineheight{1.25}\smash{\begin{tabular}[t]{l}(a)\end{tabular}}}}%
    \put(0.47345862,0.00334464){\makebox(0,0)[lt]{\lineheight{1.25}\smash{\begin{tabular}[t]{l}(b)\end{tabular}}}}%
    \put(0.84879305,0.00334464){\makebox(0,0)[lt]{\lineheight{1.25}\smash{\begin{tabular}[t]{l}(c)\end{tabular}}}}%
  \end{picture}%
\endgroup%

\caption{A family of slice knots $K_m$ and a distinguished loop $\ell$ in $S^3 \setminus K_0$.}\label{fig:simple-knot}
\end{figure}

Our goal is to produce disks that are topologically but not smoothly isotopic. When the disks' complements have infinite cyclic fundamental group, the topological isotopy is furnished by  a result of Conway and Powell. 

\begin{thm}[{\cite[Theorem~1.1]{conway-powell}}]\label{thm:iso}
Any pair of disks $D, D' \subset B^4$ with the same boundary and $\pi_1(B^4\setminus D) \cong \pi_1(B^4 \setminus D) \cong \zz$ are topologically isotopic rel  boundary. 
\end{thm}

We note that the results in \cite{conway-powell} apply to \emph{homotopy ribbon} disks, i.e., $D \subset B^4$ such that the inclusion $S^3 \setminus \partial D \hookrightarrow B^4 \setminus D$  induces a surjection on fundamental groups. When $\pi_1(B^4 \setminus D)$ is infinite cyclic, this inclusion-induced map is always surjective.

To distinguish our disks up to ambient diffeomorphism, we show that there is a knot in $ S^3 \setminus K$ that bounds a smoothly embedded, once-punctured torus  in $B^4 \setminus D'$ but not in $B^4 \setminus D$ (even up to ambient diffeomorphism). In $B^4 \setminus D$, we obstruct such a surface indirectly  using the invariant $\tau$ derived from knot Floer homology \cite{oz-sz:genus}.  To describe the obstruction, we recall that the \emph{$n$-trace} of a knot $K$ in $S^3$ is the 4-manifold $X_n(K)$ obtained by attaching an $n$-framed 2-handle to $B^4$ along $K$. The \emph{$n$-shake genus} $g^n_{sh}(K)$ of $K$ is defined to be the minimal genus of a smoothly embedded surface generating the second homology of $X_n(K)$. It can be shown that $\tau$ is a lower bound on the zero-shake genus; see Remark~4.10 or Theorem~1.6 of \cite{hmp:mazur}.

\begin{thm}[{cf~\cite{hmp:mazur}}]\label{thm:bound}
If $K$ is a knot in $S^3$, then $|\tau(K)| \leq g^0_{sh}(K)$.
\end{thm}

These tools in hand, we can produce our first examples of exotic slice disks.

\begin{thm}\label{thm:disks}
There exist infinitely many pairs of smooth embeddings of the disk into $B^4$ that are isotopic through homeomorphisms but not diffeomorphisms of $B^4$.
\end{thm}

\begin{proof}
Fix $m \in \zz$ and let $D,D' \subset B^4$ denote the corresponding slice disks depicted in Figure~\ref{fig:first-link}. We claim that $D$ and $D'$ are topologically isotopic rel boundary for all $m \in \zz$ but are smoothly inequivalent for $m=0$ and $m \ll 0$.

\begin{figure}\center
\bigskip
\def\svgwidth{.9\linewidth}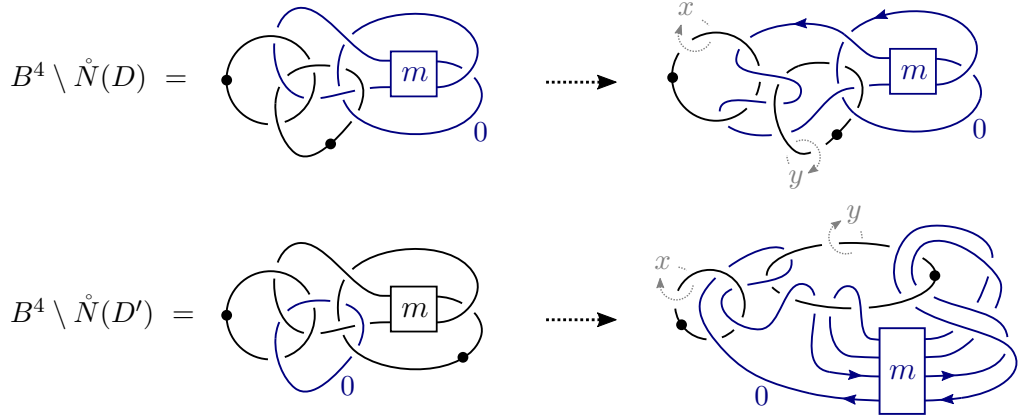
\caption{Handle diagrams for the slice disk exteriors, decorated to simplify the calculation of their fundamental groups.}\label{fig:first-exteriors}
\end{figure}

By construction, these slice disks have the same boundary $K\subset S^3$. Therefore, to establish the topological isotopy using Theorem~\ref{thm:iso}, it suffices to show  $\pi_1(B^4 \setminus D) \cong \pi_1(B^4 \setminus D') \cong \mathbb{Z}$. The disk exteriors are represented by the handle diagrams in Figure~\ref{fig:first-exteriors}. On the right, we have redrawn the diagrams and decorated them with oriented generators of the fundamental group. By tracing the 2-handle curves starting at their bottom-rightmost points, we obtain the following presentations:
\begin{align*}
\pi_1(B^4 \setminus D) &= \langle x,y \mid x y^{-1}x^{-1}y y^{-1}y^{-1}y=1 \rangle = \langle x,y \mid x y^{-1}x^{-1} =1\rangle \cong \zz \\ 
\pi_1(B^4 \setminus D') &= \langle x,y \mid x^{-1}y x y^{-1}yyy^{-1} =1\rangle= \langle x,y \mid x^{-1}y x  =1\rangle \cong \zz.
\end{align*}
Applying Theorem~\ref{thm:iso}, we see that $D$ and $D'$ are topologically isotopic rel boundary.

\begin{figure}[h]\center
\def\svgwidth{.7\linewidth}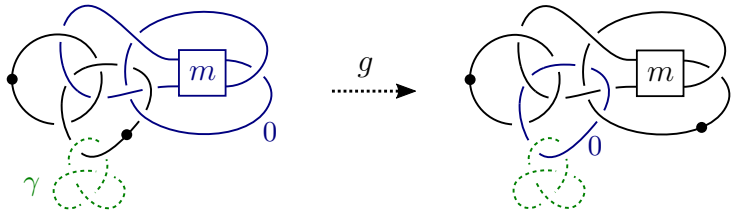
\caption{Tracking $\gamma$ under a diffeomorphism of the disk exteriors' boundaries.}
\label{fig:first-compare}
\end{figure}

For the sake of contradiction, suppose there is a diffeomorphism $f$ of $B^4$ carrying $D$ to $D'$. The restriction of $f$ to $S^3$ fixes $K$ setwise, so $f$ may be isotoped to fix a tubular neighborhood $N(K)$ and its exterior $S^3 \setminus \mathring{N}(K)$ setwise. (This isotopy is the identity on $K$ and extends to an isotopy of $B^4$ that is supported in a neighborhood of $S^3 \setminus \mathring{N}(K) \subset B^4$.)  Moreover, a closer examination of the hyperbolic geometry of $S^3 \setminus K$ shows that we may further isotope $f$ so that it restricts to the identity outside  a slightly larger tubular neighborhood of $K$; see Lemma~\ref{lem:hyperbolic} below.

Now let $g$ denote the diffeomorphism of the disk exteriors induced by $f$. If $m=0$ or $m \ll 0$, then $g$ sends the knot $\gamma$ in the boundary of $B^4 \setminus \mathring{N}(D)$ in Figure~\ref{fig:first-compare} to the knot in the boundary of $B^4 \setminus \mathring{N}(D')$ shown on the right. It is clear that the latter bounds an embedded once-punctured torus in $B^4 \setminus \mathring{N}(D')$. We claim that $\gamma$ does not bound such a surface in $B^4 \setminus \mathring{N}(D)$. To see this, we construct a new 4-manifold $X$ by attaching a pair of 2-handles to $B^4 \setminus \mathring{N}(D)$ as shown in Figure~\ref{fig:first-trace}, one of which is attached along $\gamma$. After performing the handle calculus in Figure~\ref{fig:first-trace}, we see that $X$ is the zero-trace of a knot that we will denote by $J_m$.

\begin{figure}\center
\smallskip
\smallskip
\def\svgwidth{.9\linewidth}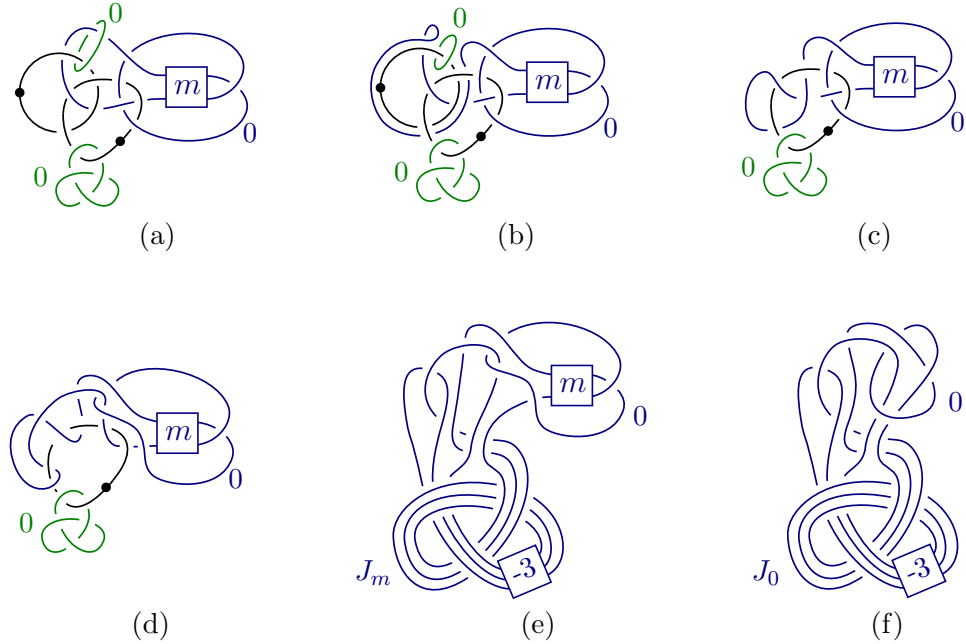
\caption{The 4-manifold $X$ from the proof of Theorem~\ref{thm:disks} is shown in part (a). Passing from (a) to (b) corresponds to a handleslide; (b)-(c) is handle cancellation and minor isotopy; (c)-(d) is isotopy; (d)-(e) consists of three handleslides and a handle cancellation, yielding the zero-trace of the knot $J_m$. Part (f) depicts $J_0$, which is obtained from $J_m$ by removing the $m$ positive twists.}\label{fig:first-trace}
\end{figure}

If $\gamma$ bounds a once-punctured torus in $B^4 \setminus \mathring{N}(D)$, then this can be capped off with the core of the 2-handle attached along $\gamma$ to yield a \emph{closed} torus  in $X=X_0(J_m)$. Looking at Figure~\ref{fig:first-trace}(a-d), it is easy to see that such a surface represents a generator of $H_2(X)$. However,  we claim that no such torus can exist. To see this, we note that $J_m$ is obtained from $J_0$ by adding $m$ positive full-twists along a pair of oppositely-oriented strands, hence $J_m$ can be turned into $J_0$ by performing $m$ positive-to-negative crossing changes. By \cite[Corollary~1.5]{oz-sz:genus} (or \cite[Corollary~3]{livingston:tau}), it follows that $\tau(J_m) \geq \tau(J_0)$ for all $m \leq 0$. Next, by a direct calculation using \cite{hfk-calc}, we calculate $\tau(J_0)=2$; see \cite{hayden:doc} for additional documentation. This implies that $\tau(J_m) \geq 2$ for all $m \leq 0$, hence Theorem~\ref{thm:bound} implies that the second homology of $X_0(J_m)$ cannot be generated by an embedded torus; this provides the desired contradiction. \end{proof}

\begin{rem*}
For exotic surfaces of all genera in $B^4$, see the proof of Theorem~\ref{thm:tbd} in \S\ref{sec:symp}.
\end{rem*}

We conclude by verifying the technical lemma invoked in the proof above.

\begin{lem}\label{lem:hyperbolic}
Let $K=K_m$ denote the knot in $S^3$ defined above for $m \in \zz$. If $m=0$ or $m \ll0$, then every self-diffeomorphism of $S^3$ carrying $K$ to itself can be assumed to fix the knot exterior $S^3 \setminus \mathring{N}(K)$.
\end{lem}

\begin{proof}
The knot complement $S^3 \setminus K$ is obtained from $S^3 \setminus K_0$ by performing $-1/m$-framed Dehn surgery along the knot $\ell$ depicted in Figure~\ref{fig:simple-knot}(c).  Using SnapPy and Sage \cite{snappy,sagemath}, we verify that $S^3 \setminus K_0$ and $S^3 \setminus (K_0 \cup \ell)$ admit  hyperbolic structures with trivial isometry groups; see \cite{hayden:doc} for additional documentation regarding this calculation. Thus for $m=0$ and $m \ll 0$, the knot complement $S^3 \setminus K$  admits a hyperbolic structure with trivial isometry group; see \cite[Lemma~2.2]{dhl}. 

By \cite{gabai:smale}, every self-diffeomorphism of a hyperbolic 3-manifold is isotopic to an isometry, hence every self-diffeomorphism of $S^3 \setminus K$ is isotopic to the identity; we remark that the earlier results in \cite{hatcher:large,ivanov} can be used in place of \cite{gabai:smale}  for the case of knot complements.  Now observe that any self-diffeomorphism of $S^3$ carrying $K$ to itself may be isotoped to fix a tubular neighborhood $N(K)$ and its exterior $S^3 \setminus \mathring{N}(K)$ setwise. (This isotopy is the identity on $K$, and we note that it extends to an isotopy of $B^4$ that is supported in a neighborhood of $S^3 \setminus \mathring{N}(K) \subset B^4$.)  Hence we may further isotope any self-diffeomorphism of $S^3$ that fixes $K$ setwise so that it restricts to the identity on $S^3 \setminus \mathring{N}(K)$ for $m=0$ or $m \ll 0$, as claimed.
\end{proof}

\vspace{-.4in}

\section{The symplectic setting}\label{sec:symp}

In this section, we adapt the topological construction from \S\ref{sec:simple} to the symplectic  setting, using the family of links shown in Figure~\ref{fig:main-link}(a). For the necessary background on contact, symplectic, and Stein manifolds,  we refer the reader to \cite{os:surgery,gompf:stein,ec:book}.

 The main theorems will be derived from the following:

\begin{thm}\label{thm:tbd}
For any integer $g \geq 0$, there exist infinitely many transverse knots in $(S^3,\xi_\st)$ that each bound a pair of properly embedded symplectic surfaces of genus $g$ in $(B^4,\omega_\st)$ that are isotopic through homeomorphisms of $B^4$ yet not through diffeomorphisms of $B^4$.
\end{thm}

\begin{figure}[h]\center
\smallskip
\def\svgwidth{.75\linewidth}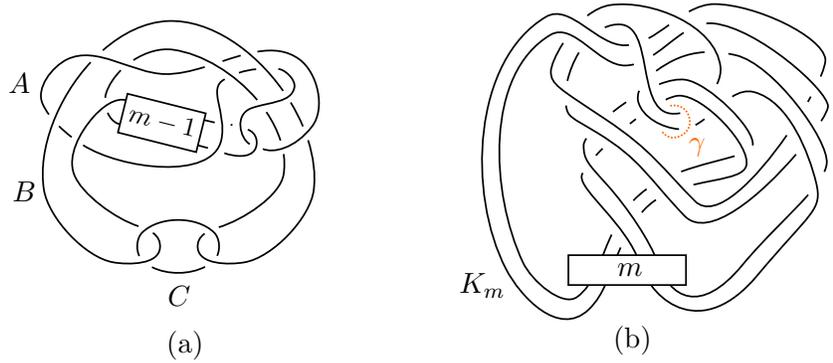
\caption{(a) A family of links satisfying the conditions from \S\ref{sec:simple}. (b) The resulting family of slice knots $K_m$. (The curve $\gamma$ will be used to study the slice disks bounded by $K_m$.)}\label{fig:main-link}
\end{figure}

Our primary tool for constructing symplectic surfaces in $(B^4,\omega_\st)$  is the following lemma.  It has likely been known to experts for some time, forming a symplectic analog of  a construction of Lagrangian cobordisms due to Ekholm-Honda-K\'alm\'an \cite{ehk:cobordisms} and Rizell \cite{rizell:surgery}; see also Theorem~4.2 and Figure~5 of \cite{bst:construct}. For a proof, see  Example~4.7 and Lemma~5.1  of \cite{hayden:stein} (cf \cite[Lemma 2.7]{etnyre-golla}).

\begin{lem}\label{lem:build}
If a smooth surface $\surf$ in $B^4$ can be presented by a sequence of transverse links  in $(S^3,\xi_\st)$ related by transverse isotopy and the diagram moves in Figure~\ref{fig:moves}, then $\surf$ is smoothly isotopic rel boundary to a symplectic surface in $(B^4,\omega_\st)$.
\end{lem}

\begin{figure}[h]\center
\def\svgwidth{.8\linewidth}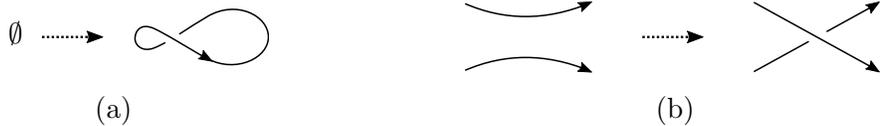
\caption{Elementary birth and saddle moves  between transverse link diagrams.}\label{fig:moves}
\end{figure}

For an example illustrating Lemma~\ref{lem:build}, see Figure~\ref{fig:torus}, which shows that the standard slice surface for the $T_{2,n}$ torus link can be realized as a symplectic surface in $(B^4,\omega_\st)$. 

\begin{figure}[b]\center
\def\svgwidth{.73\linewidth}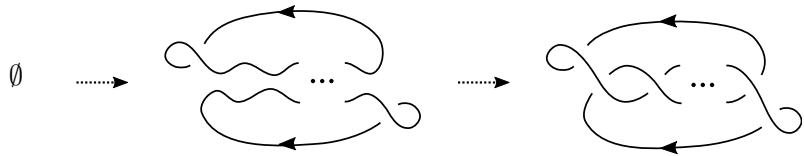
\caption{Constructing a symplectic surface bounded by a torus link.}
\label{fig:torus}
\end{figure}

\begin{rem}\label{rem:gen}
Lemma~\ref{lem:build} also holds for surfaces in a compact piece $Y \times [a,b]$ of the symplectization $Y \times \rr$ of any contact 3-manifold $Y$, where the diagram moves occur in local Darboux charts. 
\end{rem}

\begin{figure}[t]
\hspace{-.275in} \def\svgwidth{1.075\linewidth}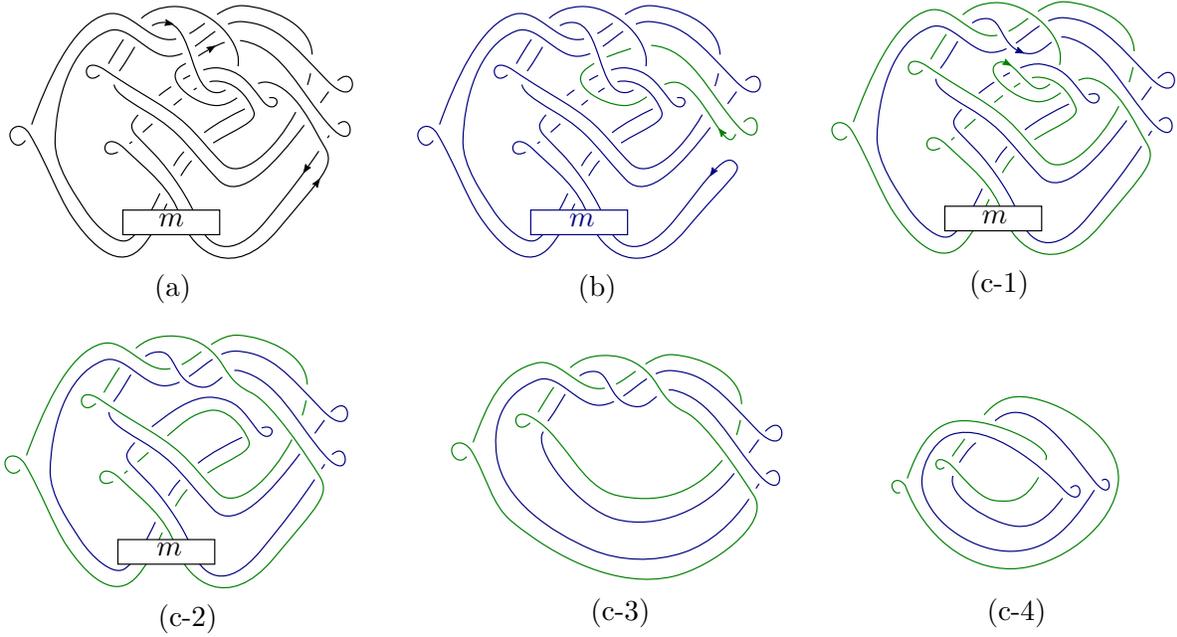
\caption{For $m\leq 0$, a transverse knot $K$ is shown in part (a). Passing from (a) to (b) or (c-1) corresponds to a transverse saddle move from Lemma~\ref{lem:build} and Figure~\ref{fig:moves}(b). It is easy to see that (b) depicts a standard two-component transverse unlink, which bounds a pair of disks via Figure~\ref{fig:moves}(a). The same is true of the link in (c-1); for this link, we include additional diagrams in parts (c-2)-(c-4) indicating the transverse isotopy that makes the standard two-component transverse unlink more apparent.}
\label{fig:new-ascending}
\end{figure}

\begin{proof}[Proof of Theorem~\ref{thm:tbd}]
For each integer $m \in \zz$, let $K=K_m$ be the knot in Figure~\ref{fig:main-link}(b), obtained by applying the construction from \S\ref{sec:simple} to the link in Figure~\ref{fig:main-link}(a). 

We begin by describing  symplectic disks in $(B^4,\omega_\st)$ bounded by the transverse representative of $K$ shown in part (a) of Figure~\ref{fig:new-ascending}. Parts (b) and (c-1)-(c-4) of Figure~\ref{fig:new-ascending}  each depict a standard two-component transverse unlink obtained from $K$ by the saddle move from Lemma~\ref{lem:build} and Figure~\ref{fig:moves}. These determine disks $D$ and $D'$, respectively, each consisting of one saddle and two minima. By Lemma~\ref{lem:build}, each of these disks is smoothly isotopic (rel boundary) to a symplectic disk in $(B^4,\omega_\st)$. We claim that
\begin{enumerate}
\item [(a)] $D$ and $D'$ are topologically isotopic (rel boundary), and

\item [(b)] for  $m=0$ and $m \ll 0$, the exteriors of $D$ and $D'$ are not diffeomorphic, so there is no diffeomorphism of $B^4$ carrying $D'$ to $D$.
\end{enumerate}

To prove (a),  it suffices to show that $B^4 \setminus D$ and $B^4 \setminus D'$ have $\pi_1 \cong \zz$,   at which point the claim follows from Theorem~\ref{thm:iso}. We find handle diagrams for the disk exteriors using the recipe from \cite{GompfStipsicz4} in Figures~\ref{fig:exterior-D}-\ref{fig:exterior-Dprime}, then compute their  fundamental groups:
\begin{align*}
\pi_1(B^4 \setminus D)  &\cong \langle x,y \mid x^{-1} y^{-1} x y^{-1} y =1 \rangle  \cong \langle x,y \mid x^{-1} y^{-1} x =1\rangle  \cong \zz 
\\
\pi_1(B^4 \setminus D') &\cong \langle x,y \mid y^{-1} y x y^{-1}  x^{-1} y y^{-1} =1 \rangle \cong \langle x,y \mid x y^{-1}  x^{-1}  =1 \rangle \cong \zz.
\end{align*}

\begin{figure}[t]\center
\def\svgwidth{.75\linewidth}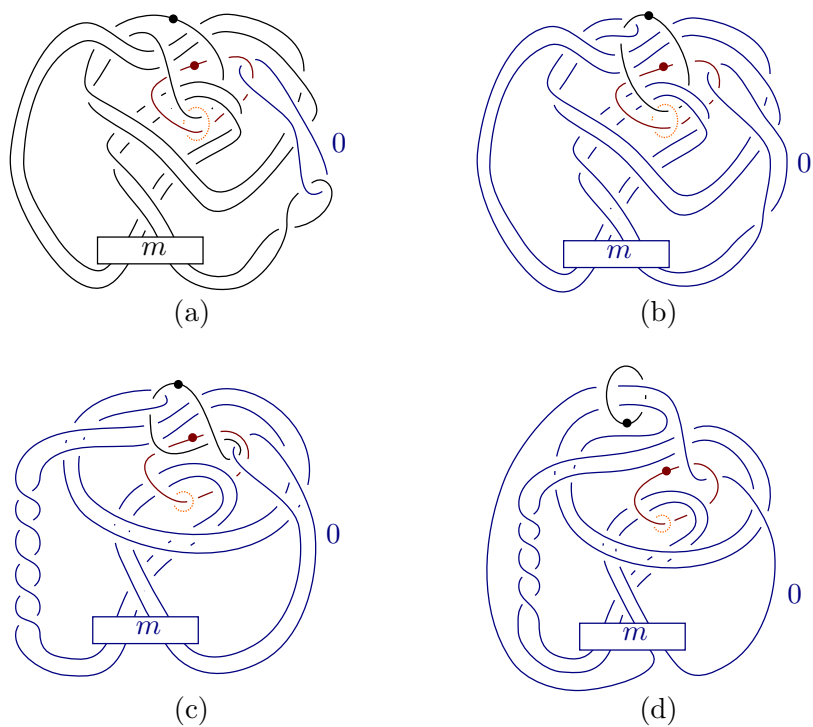
\caption{Simplifying a handle diagram for the exterior of $D$ in $B^4$. Passing from (b) to (c) corresponds to sliding one 1-handle over the other; other steps are isotopy.}
\label{fig:exterior-D}
\end{figure}

To prove (b), we assume towards a contradiction that there is a diffeomorphism of $B^4$ carrying $D'$ to $D$. To simplify our argument, we make a technical observation about the knot $K$: For $m=0$ and $m \ll 0$, every self-diffeomorphism of $S^3$ that fixes $K$ setwise  can be isotoped (rel $K$) to fix the knot exterior $S^3 \setminus \mathring{N}(K)$. The proof is analogous  to the proof of Lemma~\ref{lem:hyperbolic}, so we omit its proof here; see \cite{hayden:doc} for additional documentation regarding this calculation. Any isotopy of  $S^3 \setminus \mathring{N}(K)$ extends to an isotopy of $B^4$ supported away from $D$. Thus, for $m=0$ and $m \ll0$, we may assume that the putative diffeomorphism of $B^4$ fixes $S^3 \setminus \mathring{N}(K)$. 

 Let $\gamma$ denote the knot in $S^3 \setminus \mathring{N}(K)$ induced by the dashed curve in Figure~\ref{fig:main-link}(b). We claim  $\gamma$ bounds a disk in $B^4$ that is disjoint from $D'$. To see this, we view $\gamma$  in the disk exterior $B^4 \setminus \mathring{N}(D')$ as  in Figure~\ref{fig:exterior-Dprime}(b). This knot does not pass over any 1-handles, hence it bounds a disk in $B^4$ that is disjoint from $D'$. Since the supposed diffeomorphism  of $B^4$ fixes $\gamma$ and carries $D'$ to $D$, it carries the disk $\gamma$ bounds in $B^4 \setminus D'$ to a disk that $\gamma$ bounds in $B^4 \setminus D$.  We will prove that no such disk can exist, producing the desired contradiction. To that end, we consider $\Sigma(B^4,D)$, the branched double cover of $B^4$ along $D$.

\begin{figure}\center
\def\svgwidth{.95\linewidth}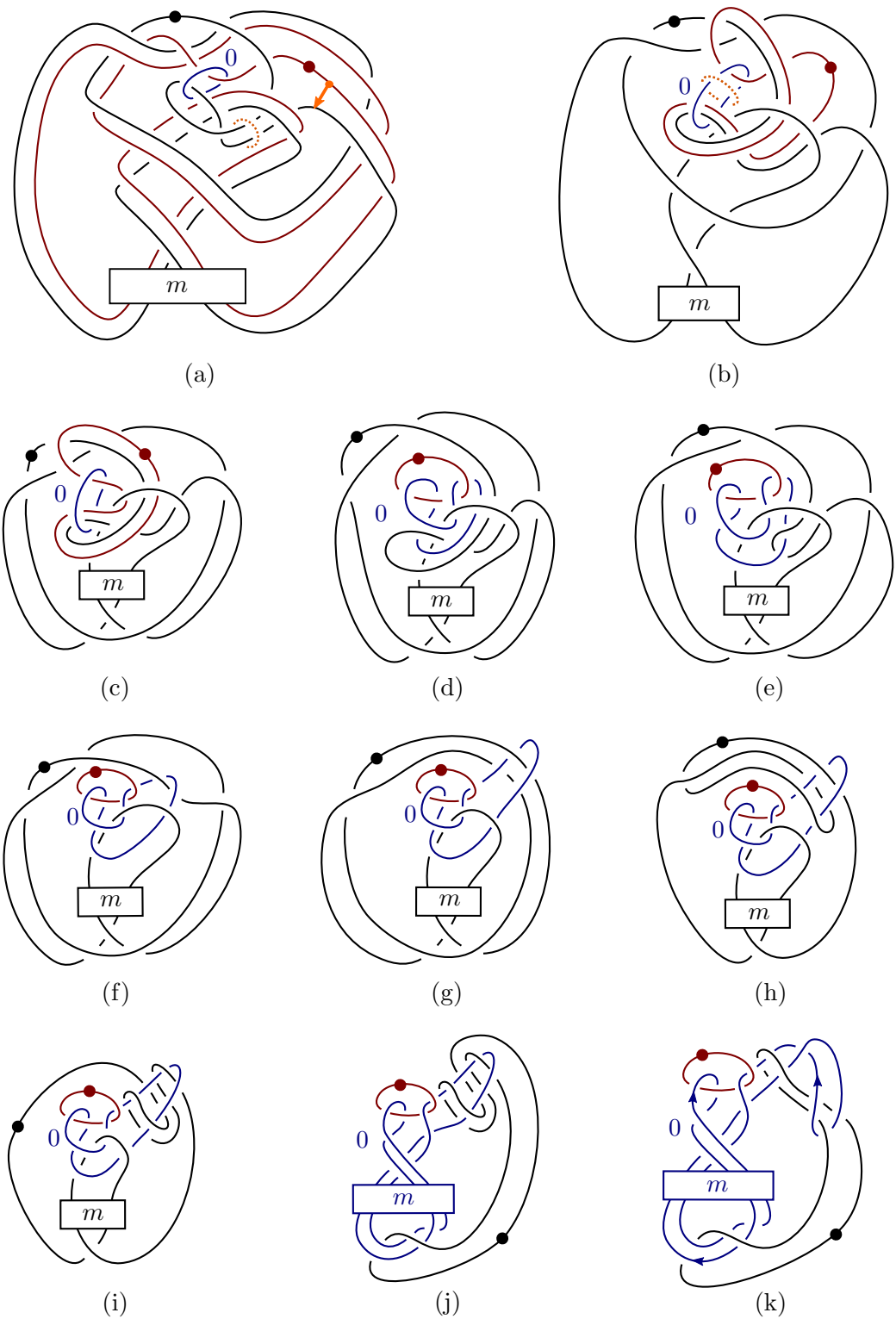
\caption{Simplifying a handle diagram for the exterior of $D'$ in $B^4$.}
\label{fig:exterior-Dprime}
\end{figure}

Following the recipe from \cite{GompfStipsicz4}, we find a handle diagram for $\Sigma(B^4,D)$. This is carried out in parts (a) through (h) of Figure~\ref{fig:dbc}. 
 Throughout parts of this figure, we track two curves: the lift $\tilde K$ of $K$ and a curve that becomes a preferred lift $\tilde \gamma$ of $\gamma$.  In Figure~\ref{fig:dbc}(i), we show that $\Sigma(B^4,D)$ admits a Stein structure by recasting  the diagram from part (h) in the ``standard form'' of \cite{gompf:stein}.

\begin{figure}\center
\def\svgwidth{.5\linewidth}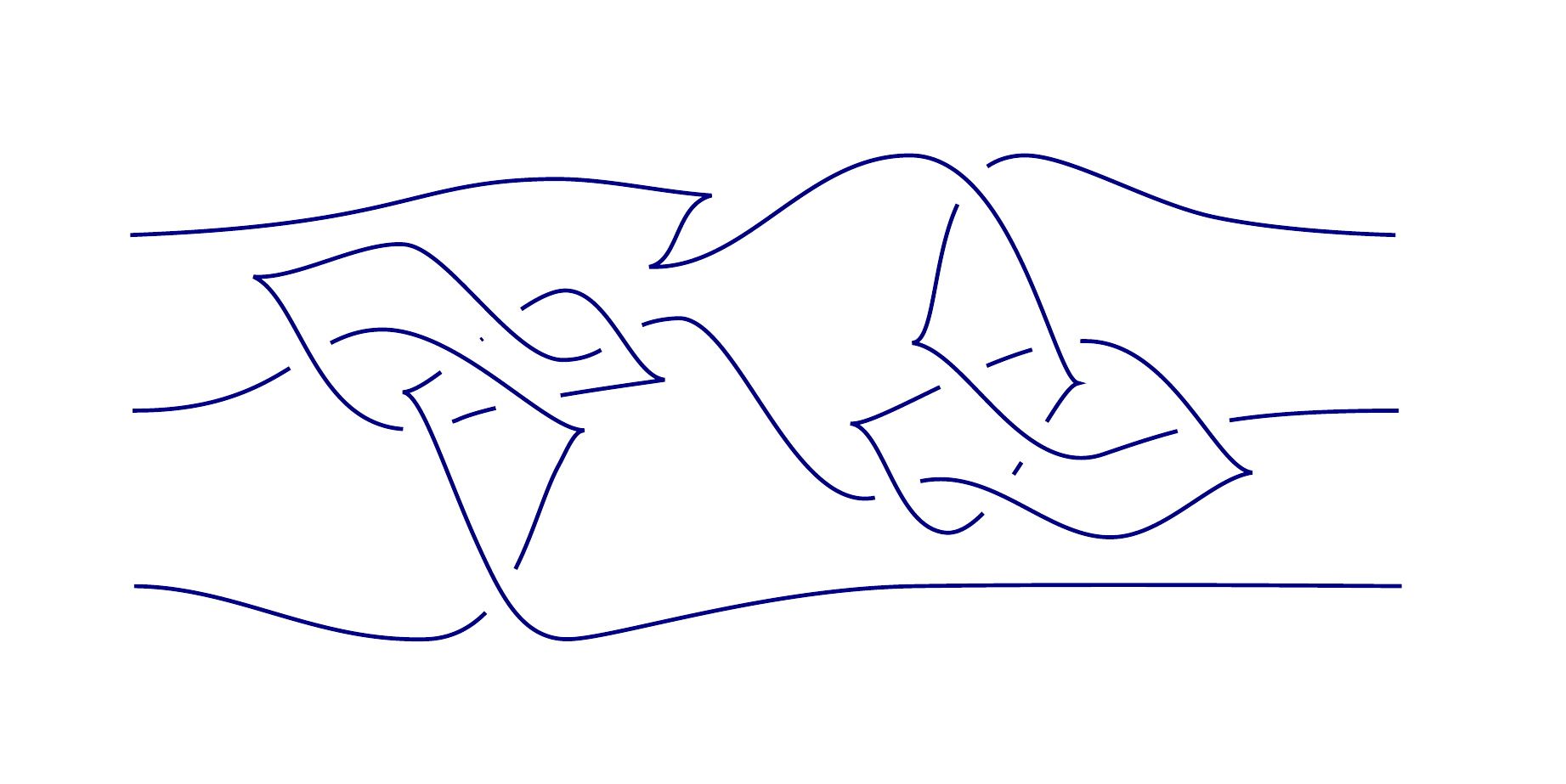
\caption{A Stein handle diagram for $\Sigma(B^4,D)$.}
\label{fig:lifts}
\end{figure}

\begin{figure}\center
\def\svgwidth{.875\linewidth}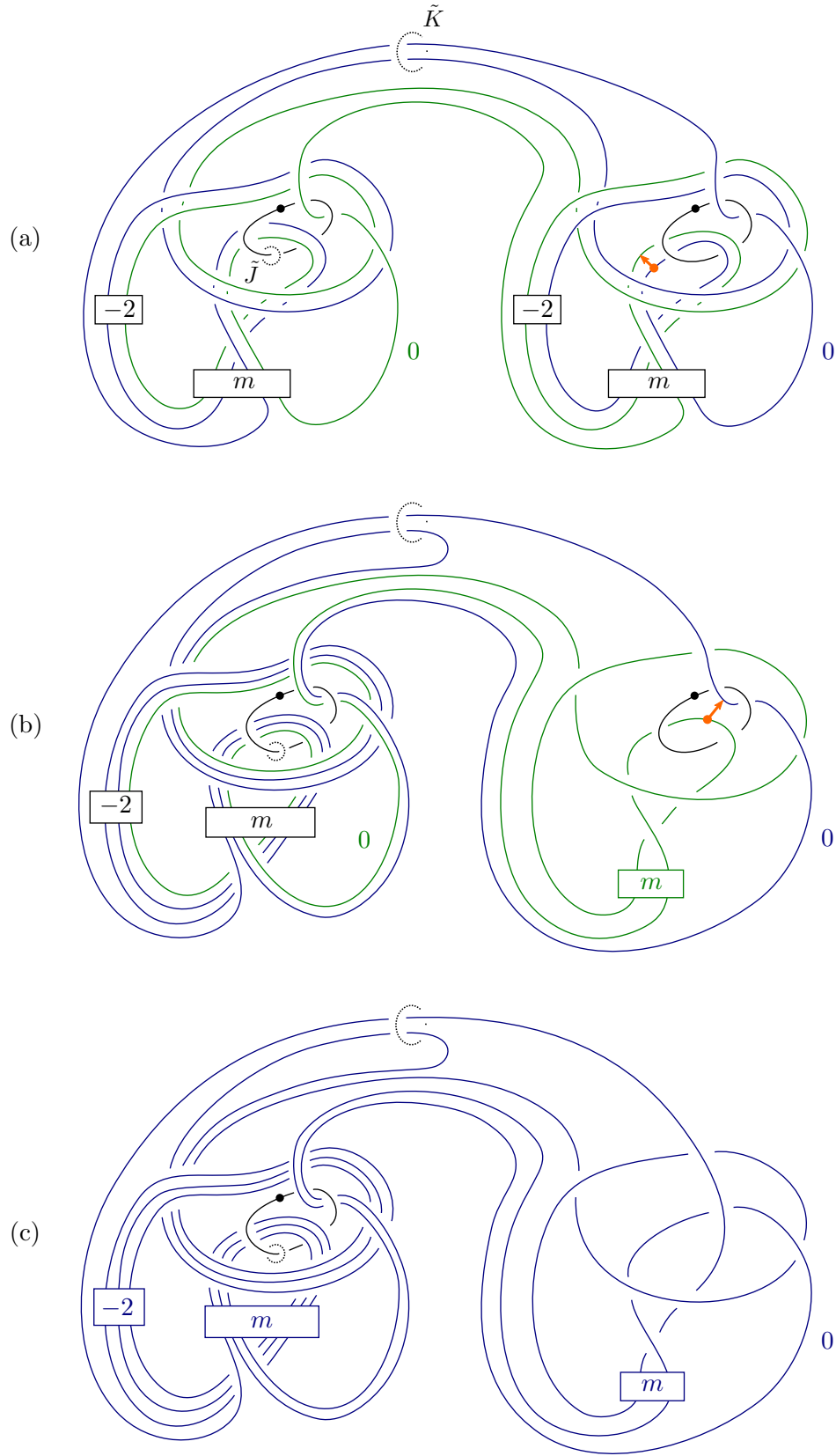
\caption{Simplifying the diagram of $\Sigma(B^4,D)$ using handle slides and cancellation.}
\label{fig:dbc}
\end{figure}

\begin{figure}\center
\def\svgwidth{\linewidth}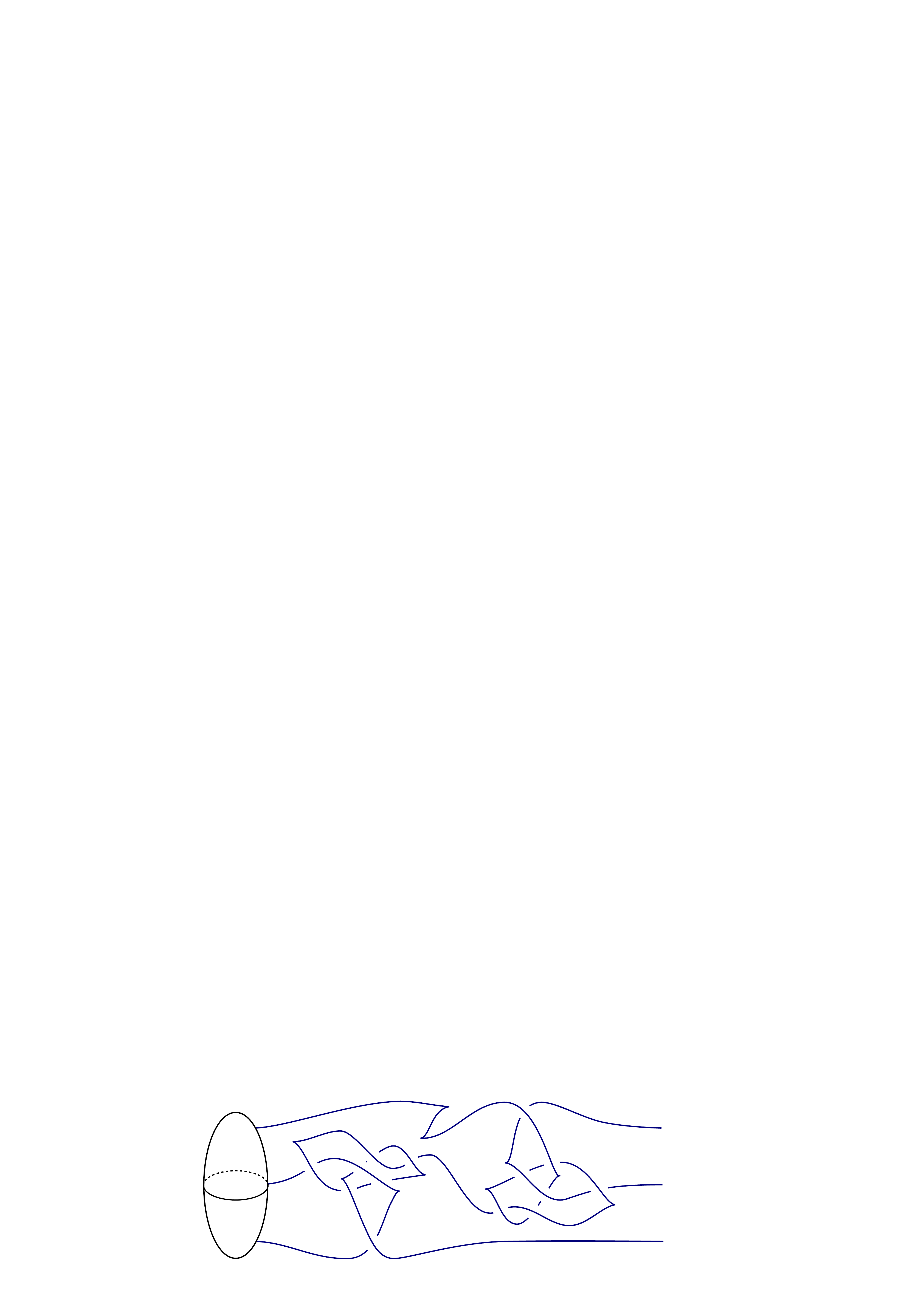
   \renewcommand\figurename{Figure}
   \renewcommand{\thefigure}{\arabic{figure} - Continued}
   \addtocounter{figure}{-1}
\caption{Further simplifying the diagram of $\Sigma(B^4,D)$ by isotopy.}
\label{fig:dbc-2}
\end{figure}

From here, our argument follows a well-tread path that dates back at least to \cite{akbulut-matveyev}. Observe that the lifted knot $\tilde \gamma$ has a Legendrian representative in $\partial \Sigma(B^4,D)$ whose Thurston-Bennequin number is zero. It follows that the 4-manifold obtained by attaching a $(-1)$-framed 2-handle to $\Sigma(B^4,D)$ along $\tilde \gamma$ also admits a Stein structure. Since the curve $\gamma$  bounds a disk in $B^4$ that is disjoint from $D$, the lift $\tilde \gamma$ bounds a disk in the branched cover $\Sigma(B^4,D)$. After attaching the $(-1)$-framed 2-handle to $\tilde \gamma$, the disk bounded by $\tilde \gamma$ gives rise to a smoothly embedded 2-sphere with self-intersection number $-1$ in the resulting Stein domain. However, this contradicts \cite[Proposition~2.2]{lisca-matic}, which states that any homologically essential 2-sphere in a Stein domain has self-intersection at most $-2$. We conclude that $\gamma$ cannot bound a smoothly embedded disk in $B^4$ that is disjoint from $D$, hence there can be no diffeomorphism of $B^4$ carrying $D'$ to $D$.

This completes the proof for surfaces of genus $g=0$. For $g \geq 1$, let $\surf_0 \subset (B^4,\omega_\st)$ be the genus $g$ symplectic surface bounded by the torus knot $T_{2,2g+1}$ constructed as in Figure~\ref{fig:torus} using Lemma~\ref{lem:build}. Then let $\surf$ and $\surf'$ be obtained by taking the boundary connected sum of $\surf_0$ with $D$ and $D'$, respectively. This can be done symplectically using the 1-handle move in Figure~\ref{fig:moves}(b), so we may assume that $\surf$ and $\surf'$ are symplectic surfaces in $(B^4,\omega_\st)$ bounded by the same transverse link $K \# T_{2,2g+1}$ in $(S^3,\xi_\st)$. Since $D$ and $D'$ are topologically isotopic rel boundary, the same is true of $\surf$ and $\surf'$. 

It is straightforward to show that  $\Sigma(B^4,\surf)=\Sigma(B^4,D \natural \surf_0)$ is the boundary  sum $\Sigma(B^4,D) \natural \Sigma(B^4,\surf_0)$ and that $\Sigma(B^4,\surf_0)$ is a Stein filling of $\Sigma(S^3,T_{2,2g+1})$, which is the lens space $L(2g+1,1)$.  From above, we also recall that   $\Sigma(B^4,D)$ admits a Stein structure with the schematic handle diagram shown in Figure~\ref{fig:lifts}. Thus we may use a Stein 1-handle to form the boundary  sum $\Sigma(B^4,D) \natural \Sigma(B^4,\surf_0)=\Sigma(B^4,\surf)$.

Again let $\gamma$ be the knot in $S^3 \setminus \mathring{N}(K) \subset \partial B^4$ corresponding to the dashed curve in Figure~\ref{fig:main-link}(b). Just as before, $\gamma$ is seen to bound a smooth disk in the exterior of $\surf'$. As above, we can show that any diffeomorphism of $B^4$ carrying $\surf'$ to $\surf$ can be assumed to fix $\gamma$; this argument is only mildly more complicated than the previous setting, and is described in Remark~\ref{rem:jsj}.  
Hence any diffeomorphism of $B^4$ carrying $\surf'$ to $\surf$ carries the disk bounded by $\gamma$ in $B^4 \setminus \surf'$ to a disk bounded by $\gamma$ in $B^4 \setminus \surf$. 

To obstruct the existence of such a disk, we again consider a lift $\tilde \gamma$ of $\gamma$ to the double branched cover $\Sigma(B^4,\surf)$.  As before, the lift $\tilde \gamma$ has a Legendrian representative in the boundary of the Stein domain $\Sigma(B^4,\surf)$ with Thurston-Bennequin number zero. If $\gamma$ bounds a disk in $B^4 \setminus \surf$, then $\tilde \gamma$ bounds a disk in $\Sigma(B^4,\surf)$. Attaching a $(-1)$-framed 2-handle to $\Sigma(B^4,\surf)$ along $\tilde \gamma$ then yields a Stein domain containing a smoothly embedded 2-sphere with self-intersection $-1$, contradicting  \cite[Proposition~2.2]{lisca-matic}. We conclude that there is no diffeomorphism of $B^4$ taking $\surf'$ to $\surf$. \end{proof}

\begin{rem}\label{rem:jsj}
To study the diffeomorphisms of $S^3 \! \setminus \! (K \# T_{2,2g+1})$, we view $K \# T_{2,2g+1}$ as a satellite knot lying inside $N(K)$. The  incompressible torus $T=\partial N(K)$ separates $S^3 \setminus (K \# T_{2,2g+1})$ into two pieces, one of which is $S^3 \setminus K$. The other piece, denoted $M$, is diffeomorphic to the complement of the link in $S^3$ formed from the torus knot $T_{2,2g+1}$ and its meridian. Thus $M$ further splits into two Seifert fibered pieces: the complements  of $T_{2,2g+1}$ and the Hopf link. It follows that $S^3 \setminus K$ is the only hyperbolic piece in the JSJ decomposition \cite{jaco-shalen,johannson} of $S^3 \setminus \partial F$.  Every self-diffeomorphism of $\partial X$ can be isotoped to preserve $T$ setwise and thus preserve these distinct JSJ pieces on each side of $T$. For $m=0$ and $m \ll 0$, we already know that every self-diffeomorphism of $S^3 \setminus K$ is isotopic to the identity, hence we may assume that any self-diffeomorphism of $B^4$ carrying $\surf$ to $\surf'$  restricts to the identity on $S^3 \setminus \mathring{N}(K)$. 
\end{rem}

\textbf{Exotic closed symplectic surfaces.}  To produce \emph{closed} symplectic surfaces, we use a construction of symplectic 2-handles due to Gay \cite{gay:2-handles}, as applied by Bowden \cite{bowden:thesis}. Recall that symplectic 4-manifold $(X,\omega)$ is said to be \emph{convex} if there exists a  vector field $v$ defined on a neighborhood of $\partial X$ that is Liouville for $\omega$ (i.e.,~$\mathcal{L}_v \omega=\omega$) and points outward along $\partial X$. The pair $(X,\omega)$ is \emph{weakly convex} if $\partial X$ admits a (positive) contact structure $\xi$ such that $\omega|_{\xi}$ is nondegenerate. 

The following result combines Theorem 1.1 and Proposition 1.8 of \cite{gay:2-handles}.

\begin{thm}[Gay]
\label{thm:2-h}
Let $K$ be a transverse knot in the boundary of a convex symplectic 4-manifold $(X,\omega)$, and let $X'$ be obtained from $X$ by attaching a 2-handle along $K$. If the framing of the 2-handle is sufficiently negative, then $X'$ admits a symplectic form $\omega'$ such that $(X',\omega')$ is weakly convex. 

 In particular, if $K$ is the (positive) transverse push-off of a Legendrian knot, then the conclusion above holds for any framing that is strictly negative relative to the Thurston-Bennequin framing of this Legendrian representative. 
\end{thm}

The core disks of these symplectic 2-handles are themselves symplectic and can be used to cap off symplectic surfaces $\surfb$ in $(X,\omega)$ bounded by $K$. To ensure that the resulting symplectic surfaces are smooth, we require that $\surfb$ be \emph{cylindrical near the boundary}. This means that $\surfb$ must be tangent to a Liouville vector field $v$ for $\omega$ in a neighborhood of $\partial X$. It is straightforward to arrange that any surface constructed using Lemma~\ref{lem:build} is cylindrical near its boundary.

\begin{prop}[{cf \cite[\S7.2]{bowden:thesis}}]\label{prop:cap}
Let $K$, $(X,\omega)$, and $(X',\omega')$  be as above. Any symplectic surface $\surfb \subset X$ with $\partial \surfb=K$ that is cylindrical near its boundary can be capped off with the core disk of the symplectic 2-handle to yield a closed symplectic surface $\surfb'$ in $X'$.
\end{prop}

\begin{rem} The results above (Theorem~\ref{thm:2-h} and Proposition~\ref{prop:cap}) can be extended using ``higher genus'' handles of the form $F \times D^2$ where $F$ is a compact, orientable surface of any genus with one boundary component, attached to $X$ using an embedding $\partial F \times D^2 \hookrightarrow \partial X$; see  \cite[\S3]{hp:embedding}.
\end{rem}

Ideally, we aim to work with convex symplectic 4-manifolds, whereas the construction in Theorem~\ref{thm:2-h} only ensures \emph{weakly} convex boundary. Fortunately, when the boundary is a rational homology 3-sphere, the stronger convexity condition can always be achieved:

\begin{lem}[{\cite[Lemma~1.1]{oo:simple}}]\label{lem:QHS3}
If a symplectic 4-manifold $(X,\omega)$ is weakly convex and $\partial X$ is a rational homology 3-sphere, then $\omega$ may be deformed in an arbitrarily small neighborhood of $\partial X$ to a symplectic form $\omega'$ such that $(X,\omega')$ is convex.
\end{lem}

Before constructing our examples of exotically knotted, closed symplectic surfaces, we also require some topological preliminaries about branched covers of knot traces. Recall from \S\ref{sec:simple} that the \emph{$n$-trace} of a knot  $K$ in $S^3$ is the 4-manifold $X=X_n(K)$ obtained by attaching an $n$-framed 2-handle to $B^4$ along $K$. Given a pair of slice surfaces $F,F' \subset B^4$ bounded by the same knot $K$, we obtain surfaces $S,S' \subset X$ by capping off each of $F,F' \subset B^4$ with the core of the 2-handle in $X$. Our goal will be to distinguish $S$ and $S'$ by studying surfaces in their complements.  In our case, this will be simplified by passing to the double branched covers of $X$ along $S$ and $S'$. The next lemma describes the handle structure of such branched covers.

\begin{lem}\label{lem:cover}
Let $F$ be a properly embedded surface in $B^4$ bounded by $K$ in $S^3$ and, for any $n \in \zz$, let $S$ be the surface in the knot trace $X=X_n(K)$ obtained by capping off $F$ with the core of the 2-handle. Then 
\begin{enumerate}
\item [(a)] $H_1(X \setminus S)$ is generated by the meridian to $S$ and is isomorphic to $ \zz/n$, and 
\item [(b)] for any positive integer $k$ dividing $n$, the $k$-fold cyclic branched cover of $X$ along $S$, denoted $\Sigma_k(X,S)$, is obtained from the $k$-fold cyclic branched cover of $B^4$ along $F$, denoted $\Sigma_k(B^4,F)$, by attaching a 2-handle to the lift $\tilde K$ in $\Sigma_k(S^3,K)=\partial \Sigma_k(B^4,F)$ with framing $n/k$.
\end{enumerate}
\end{lem}

\begin{proof}
For part (a), let $\mu$ and $\lambda$ denote a meridian and Seifert longitude to $K$, respectively. Since $H_1(X)=0$, it is easy to see that every element of $H_1(X \setminus S)$ is homologous to a multiple of $[\mu]$, hence $[\mu]$ generates $H_1(X\setminus S)$. On the other hand, the 2-handle framing curve $\lambda+n \cdot \mu$ bounds a pushoff of the 2-handle's core disk, so $n[\mu]=[\lambda+n \cdot \mu]=0$ in $H_1(X \setminus S)$ because $\lambda$ is nullhomologous. It follows that $H_1(X \setminus S) \cong \zz/n$.

To prove part (b), we let $k$ be any positive integer dividing $n$. The meridian of $S$ generates $H_1(X \setminus S) \cong \zz/n$, so we may construct a $k$-fold cyclic branched cover of $X$ along $S$.  Since $X$ is formed from the union of $B^4$ and $D^2 \times D^2$, the branched cover $\Sigma_k(X,S)$ is  a union of the branched covers $\Sigma_k(B^4,F)$ and $\Sigma_k(D^2 \times D^2, D^2 \times 0)$. The latter is simply $D^2 \times D^2$, with the branched covering map given by $(z,w)\mapsto (z,w^k)$ in complex coordinates on $D^2 \times D^2 \subset \cc \times \cc$. Similarly, the gluing region $N(K) \cong S^1 \times D^2$ lifts to a neighborhood $N(\tilde K) \cong S^1 \times D^2$ of the lifted knot $\tilde K$ in $\Sigma_k(S^3,K)=\partial \Sigma_k(B^4,F)$. It follows that $\Sigma_k(X,S)$ is obtained from $\Sigma_k(B^4,F)$ by attaching a 2-handle along $\tilde K$ in $\Sigma_k(S^3,K)$. To determine the framing, we note that the meridian of $\tilde K$ covers the meridian of $K$ with degree $k$  (whereas $\tilde K$ has a preferred Seifert longitude that covers the Seifert longitude of $K$ with degree one). Thus the $n$-framing curve of $K$ in $S^3$ lifts to an $n/k$-framing curve of $\tilde K$ in $ \Sigma_k(S^3,K)$. Therefore the disks $D^2 \times \{pt\}$ in the lifted 2-handle meet $N(\tilde K)$ along $n/k$-framing curves of $\tilde K$, i.e.,~the lifted 2-handle is attached to $\Sigma_k(B^4,F)$ along $\tilde K$ with framing $n/k$.
\end{proof}

We now produce  examples of closed symplectic surfaces that are exotically knotted.

\begin{proof}[Proof of Theorem~\ref{thm:symp}] 
We first prove the theorem in the case $g=0$, i.e.,~pairs of exotic symplectic 2-spheres. For $m\leq 0$, let $K$ denote the knot bounding the  symplectic disks $D,D' \subset B^4$ constructed in the proof of Theorem~\ref{thm:tbd}. Our ambient 4-manifold $X$ will be the knot trace $X_n(K)$ for a choice of integer $n \leq 0$, and the 2-spheres $S,S'\subset X$ are obtained by capping off the disks $D,D' \subset B^4$ with the core of the 2-handle in $X$. 

We claim that
\begin{enumerate}
\item [(a)] $S$ and $S'$ are topologically isotopic,
\item [(b)] for any integer $n \leq -3$,  the 4-manifold $X=X_n(K)$ admits a convex symplectic structure with respect to which $S$ and $S'$ are symplectic, and
\item [(c)] for \emph{even}  $n \ll 0$ and either $m=0$ or $m \ll 0$, the exteriors of $S$ and $S'$ are not diffeomorphic, so there is no diffeomorphism of $X$ carrying $S'$ to $S$.
\end{enumerate}

The claim in (a) follows immediately because $S$ and $S'$ are obtained by capping off $D$ and $D'$ with the core of the 2-handle attached along their common boundary $K$.  

To prove (b), we note that $K$ can be realized as the (positive) transverse pushoff of a Legendrian knot $L$ with Thurston-Bennequin number $tb(L)=-2$; see Figure~\ref{fig:pushoff}. After negative Legendrian stabilization of $L$ (inducing transverse isotopy of $K$), we can ensure $tb(L)-1=n$ for any $n \leq -3$. Apply Theorem~\ref{thm:2-h} to produce a weakly convex symplectic structure on $X$ and apply Proposition~\ref{prop:cap} to realize $S$ and $S'$ symplectically. The boundary $\partial X = S^3_n(K)$ is a rational homology 3-sphere (because $n \neq 0$), so, by Lemma~\ref{lem:QHS3}, the symplectic structure can be made convex using a deformation supported near $\partial X$.

\begin{figure}\center
\def\svgwidth{.8\linewidth}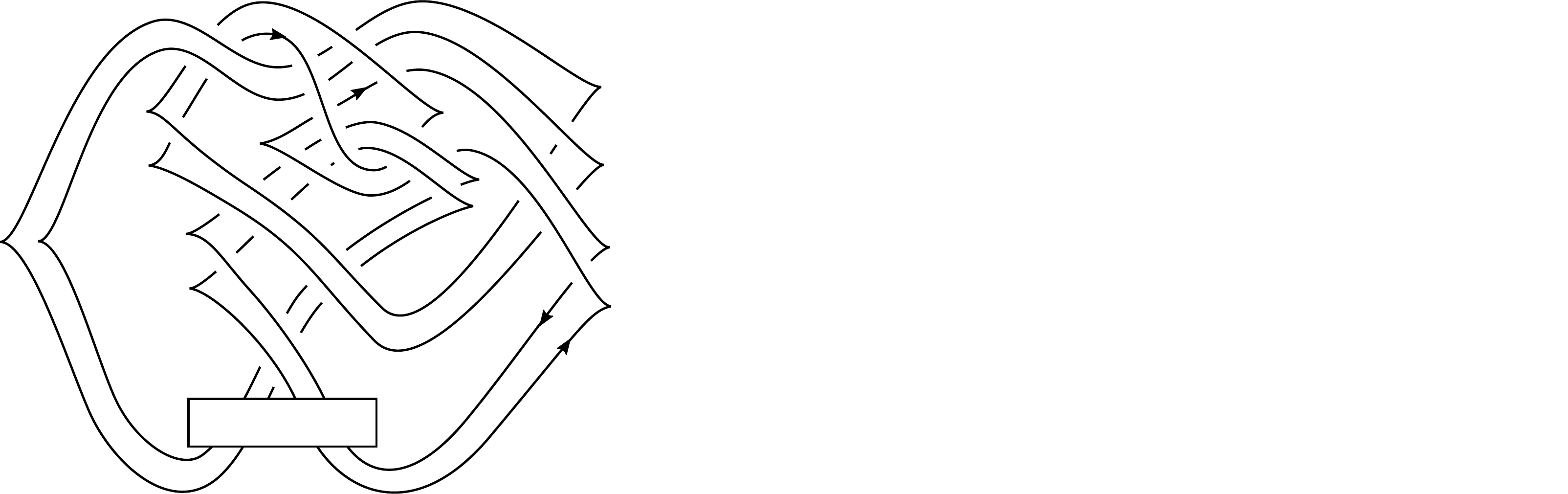
\caption{The transverse knot $K$ (right) is the pushoff of a Legendrian knot $L$ (left) with Thurston-Bennequin number $tb(L)=-2$ and rotation number $r(L)=-1$.}
\label{fig:pushoff}
\end{figure}

To prove (c), we again make a simplifying technical observation about the boundary of our 4-manifold:  For $n \ll 0$ and either $m=0$ or $m \ll 0$, every self-diffeomorphism of $\partial X$ is isotopic to the identity. This follows the same strategy as in the proofs of Theorem~\ref{thm:tbd} and Lemma~\ref{lem:hyperbolic}. In particular, as in the proof of Theorem~\ref{thm:tbd}, the calculations in \cite{hayden:doc} show that $S^3 \setminus K$ admits a hyperbolic structure with trivial isometry group. As in the proof of Lemma~\ref{lem:hyperbolic}, we may then apply \cite[Lemma~2.2]{dhl} and \cite{gabai:smale} to conclude that every self-diffeomorphism of $\partial X = S^3_n(K)$ is isotopic to the identity for $n \ll 0$. 

We also recall that, when viewed in $S^3 \setminus K$, the knot $\gamma$ from Figure~\ref{fig:main-link}(b) bounds a disk in $B^4$ that is disjoint from $D'$. It follows that, when viewed in $\partial X$, $\gamma$ also bounds a disk in $X$ that is disjoint from $S'$. 

For the sake of contradiction, assume that there exists a diffeomorphism of $X$ carrying $S'$ to $S$. For $m=0$ or $m \ll 0$, the diffeomorphism can be assumed to fix $\partial X$.  Since this diffeomorphism carries $S$ to $S'$ and fixes $\gamma \subset \partial X$, it carries the disk that $\gamma$ bounds in $X \setminus S'$ to a disk that $\gamma$ bounds in $X \setminus S$. We again show that no such disk can exist by considering $\Sigma(X,S)$, the branched double cover of $X$ along $S$. Per Lemma~\ref{lem:cover}, $\Sigma(X,S)$ is obtained from $\Sigma(B^4,D)$, the branched double cover of $B^4$ along $D$, by attaching a 2-handle to the lift $\tilde K$ of $K=\partial D$ with framing $n/2$.

A handle diagram for $\Sigma(X,S)$  is obtained by attaching an $n/2$-framed 2-handle to the lift $\tilde K$ in Figure~\ref{fig:lifts}. Though it is not drawn in detail, the knot $\tilde K$ can be represented by \emph{some} Legendrian knot in standard form in the diagram;  the attaching curve for the 0-framed 2-handle and the dashed curve representing  $\tilde \gamma$ are drawn as shown, though $\tilde K$ may have crossings with these curves. Thus, for all even integers $n$ such that $n/2$ is less than the Thurston-Bennequin number of the chosen Legendrian representative of $\tilde K$, we see that  $\Sigma(X,S)$ admits a Stein structure. (For $n/2< tb-1$, the Legendrian representative of $\tilde K$ will need to be stabilized before the 2-handle is attached.)

The rest of the argument from Theorem~\ref{thm:tbd} for $g=0$ carries over in the obvious way, showing that $\tilde \gamma$ cannot bound a disk in $\Sigma(X,S)$. It follows that $\gamma$ cannot bound a disk in $X \setminus S$, so there is no diffeomorphism of $X$ carrying $S'$ to $S$.

 For $g \geq 1$, our modifications to the above argument are parallel to those used to prove  Theorem~\ref{thm:tbd} for $g \geq 1$, so we only outline the changes. Let $F,F' \subset (B^4,\omega_\st)$ denote the higher genus symplectic surfaces bounded by $K \# T_{2,2g+1}$ that were constructed in the proof of Theorem~\ref{thm:tbd}.  Attaching an $n$-framed 2-handle along $K\# T_{2,2g+1}$ yields the desired ambient 4-manifold $X$, and we obtain closed, topologically isotopic surfaces $S,S' \subset X$ of genus $g$ by capping off $F,F' \subset B^4$ with the core of the 2-handle in $X$. As above, for $n \ll 0$, we may assume that $X$ admits a convex symplectic structure for which $S$ and $S'$ are symplectic. And, as before, we may simplify our analysis by showing that every self-diffeomorphism of $\partial X$ fixes $\gamma$ up to isotopy; see Remark~\ref{rem:jsj-2}.

Note that $X$ has a handle diagram obtained from Figure~\ref{fig:main-link}(b) by replacing $K$ with $K \# T_{2,2g+1}$ and attaching a $n/2$-framed 2-handle. Just as before, the knot  $\gamma$ is seen to bound a smooth disk in the exterior of $S'$. If there exists a diffeomorphism of $X$ carrying $S'$ to $S$, then it can be assumed to fix $\gamma \subset S^3 \setminus \mathring{N}(K)$, so it carries the disk bounded by $\gamma$ in $X \setminus S'$ to a disk bounded by $\gamma$ in $X \setminus S$. To obstruct the existence of such a disk, we again consider a lift $\tilde \gamma$ of $\gamma$ to the double branched cover $\Sigma(X,S)$. 

By Lemma~\ref{lem:cover}, $\Sigma(X,S)$ is obtained from $\Sigma(B^4,F)$ by attaching an $n/2$-framed 2-handle along the lift of $K \# T_{2,2g+1}$. As noted before, we may then fix a Legendrian representative of a lift of $K \# T_{2,2g+1}$, so the 4-manifold $\Sigma(X,S)$ admits a Stein structure whenever $n/2$ is less than the Thurston-Bennequin number of this lift. 

As before, the lift $\tilde \gamma$ has a Legendrian representative with Thurston-Bennequin number zero. If $\gamma$ bounds a disk in $X \setminus S$, then $\tilde \gamma$ bounds a disk in $\Sigma(X,S)$. Attaching a $(-1)$-framed 2-handle to $\Sigma(X,S)$ along $\tilde \gamma$ then yields a Stein domain containing a smoothly embedded 2-sphere with self-intersection $-1$, contradicting  \cite[Proposition~2.2]{lisca-matic}. We conclude that there is no diffeomorphism of $X$ taking $S'$ to $S$. \end{proof}

\begin{rem}\label{rem:jsj-2}
In the preceding proof for $g \geq 1$, observe that $\partial N(K)$ induces a torus $T \subset \partial X$ separating $\partial X$ into two pieces, one of which is $S^3 \setminus K$. We claim that the other piece, denoted $M$, is diffeomorphic to $S^3 \setminus T_{2,2g+1}$. To see this, observe that $M$ is obtained from the solid torus $N(K)$ by longitudinal Dehn surgery along the connected sum of the core curve $K$ and $T_{2,2g+1}$. Equivalently, $M$ can be obtained from an integral surgery on $T_{2,2g+1} \subset S^3$ by removing a meridian of $T_{2,2g+1}$. After  surgery, this meridian is isotopic to the core of the surgered solid torus, so its complement $M$ is $S^3 \setminus T_{2,2g+1}$. It follows that $T$ is incompressible and separates $\partial X$ into the two pieces of its JSJ decomposition \cite{jaco-shalen,johannson}: a hyperbolic piece $S^3 \setminus K$ and a Seifert-fibered piece $S^3 \setminus T_{2,2g+1}$. Every self-diffeomorphism of $\partial X$ can be isotoped to preserve $T$ setwise and thus preserve these distinct JSJ pieces on each side of $T$. For $m=0$ and $m \ll 0$, we already showed that every self-diffeomorphism of $S^3 \setminus K$ is isotopic to the identity, hence we may assume that any self-diffeomorphism of $\partial X$ restricts to the identity on $S^3 \setminus \mathring{N}(K) \subset \partial X$.
\end{rem}

\vspace{-.4in}

\section{The complex setting}\label{sec:complex}

This section adapts the above construction to the complex setting and then completes the proofs of three main results, namely Theorems \ref{thm:holomorphic}, \ref{thm:open}, and \ref{thm:stein}.

\vspace{-.1in}

\subsection{Braided surfaces and braid factorizations.} \label{subsec:factor} 

To adapt the method from \S\ref{sec:simple} to the complex setting, we require a construction of complex curves in $B^4 \subset \cc^2$ that gives fine control over the curve's smooth isotopy type.  We  will construct such curves using braid factorizations of \emph{quasipositive braids}, which are  products of conjugates $w \sigma_i w^{-1}$ of the standard positive Artin generators $\sigma_i$ in the braid group $B_n$ \cite{rudolph:qp-alg}.

Each quasipositive factorization of a quasipositive $n$-braid  determines a ribbon-immersed surface in $S^3$ formed from $n$ parallel disks by attaching a positively twisted band for each term $w \sigma_i w^{-1}$; see, for example, \cite[\S2]{rudolph:braided-surface}. The corresponding embedded surface in $B^4$ is known as  a  \emph{positively braided surface}. In \cite{rudolph:qp-alg} (as interpreted in \cite[\S4]{rudolph:braided-surface}), Rudolph showed that each such surface is isotopic to a compact piece of a smooth algebraic curve in $B^4 \subset \cc^2$ (where the isotopy of the surface restricts to braid isotopy along its boundary).

Key to our construction is the fact that inequivalent   quasipositive factorizations of a fixed quasipositive braid can yield inequivalent braided surfaces, hence inequivalent complex curves. For example, with some effort, the dedicated reader can use the techniques from \S\ref{sec:simple}-\ref{sec:symp} to show that the following braid factorizations define a pair of exotically knotted punctured tori that are holomorphically embedded in $B^4 \subset \cc^2$:
\begin{align*}
\beta&=(\sigma_{2}\sigma_{3}\sigma_{2}^{-1})(\sigma_{4}^{-1}\sigma_{3}^{-1}\sigma_{1}^{-2}\sigma_{2}\sigma_{3}\sigma_{2}^{-1}\sigma_{1}^2\sigma_{3}\sigma_{4})(\sigma_{1}^{-1}\sigma_{2}\sigma_{3}\sigma_{2}^{-1}\sigma_{1})\sigma_1^2(\sigma_{3}\sigma_{4}\sigma_{3}^{-1})\\
\beta'&=\sigma_{2}(w \sigma_{2}^{-1}\sigma_{1}\sigma_{2}w^{-1})(w\sigma_{2}^{-1}\sigma_{3}\sigma_{1}\sigma_{2}\sigma_{1}^{-1}\sigma_{3}^{-1}\sigma_{2}w^{-1}) (w\sigma_{3}^2\sigma_{4}\sigma_{3}^{-2}w^{-1})\sigma_1^2,
\end{align*}
where $w = \sigma_3 \sigma_4^{-1} \sigma_1^{-1} \sigma_3^{-2} \sigma_2^{-1}   \sigma_1^{-1} \sigma_3^{-1}$. 

However, while explicit braid factorizations are useful for constructing complex curves, they may be less convenient for showing that the underlying surfaces are exotically knotted.  For this reason, to prove Theorem~\ref{thm:holomorphic} below, we use results of Boileau-Orevkov \cite{bo:qp}  to indirectly construct positively braided surfaces associated to the exotic symplectic surfaces constructed in \S\ref{sec:symp}.

\vspace{-.1in}

\subsection{From symplectic surfaces to complex curves.}\label{subsec:symp-to-complex} As indicated above, we now realize the exotic symplectic surfaces in $B^4$ from Theorem~\ref{thm:tbd} as complex curves.

\begin{proof}[Proof of Theorem~\ref{thm:holomorphic}]
By \cite[Theorem~1]{bo:qp}, any symplectic surface in $(B^4,\omega_\st)$ with transverse boundary $K$ in $(S^3,\xi_\st)$ is diffeomorphic to a positively braided surface. In fact, the proof of \cite[Theorem~1]{bo:qp} shows that the diffeomorphism may be constructed so that the boundary of the braided surface is any chosen braid that is transversely isotopic to $K$ when viewed in $(S^3,\xi_\st)$. As discussed in \S\ref{subsec:factor}, Rudolph showed that every positively braided surface in $B^4$ is isotopic to a compact piece of a smooth algebraic curve \cite{rudolph:qp-alg,rudolph:braided-surface}.

Applying this to the symplectic surfaces $F$ and $F'$ in $B^4$ constructed in the proof of Theorem~\ref{thm:tbd}, we obtain holomorphic curves $C$ and $C'$  in $B^4 \subset \cc^2$ that are diffeomorphic to $F$ and $F'$, respectively. Moreover, $\partial C$ and $\partial C'$  are transversely isotopic to $K=\partial F=\partial F'$. Observe that there is a homeomorphism of $B^4$ carrying $C$ to $C'$. By Alexander's trick, any homeomorphism of $B^4$ is isotopic to the identity, so we further conclude that $C$ and $C'$ are topologically isotopic.
\end{proof}

Restricting Theorem~\ref{thm:holomorphic} to the case $g=0$, we obtain exotic holomorphic disks in $B^4 \subset \cc^2$.  We now prove that the double branched covers of these disks yield  exotic contractible Stein domains with the same contact boundary.

\begin{proof}[Proof of Theorem~\ref{thm:stein}] 
Let $C$ and $C'$ denote holomorphic disks in $B^4 \subset \cc^2$ constructed as in the preceding proof.  Since these disks are holomorphic, their branched double covers $W$ and $W'$ are Stein domains. Moreover, the transverse links $\partial C$ and $\partial C'$ in $(S^3,\xi_\st)$ are transversely isotopic to $K$, so the Stein domains $W$ and $W'$ fill isotopic contact structures on $\Sigma(S^3,K)$. And since $C$ and $C'$ are topologically isotopic, $W$ and $W'$ are homeomorphic. To distinguish $W$ and $W'$, we note that they are diffeomorphic to $\Sigma(B^4,D)$ and $\Sigma(B^4,D')$, respectively, where $D$ and $D'$ are  disks constructed  as in the proof of Theorem~\ref{thm:tbd}.  That proof shows that $\Sigma(B^4,D)$ and $\Sigma(B^4,D')$ are \emph{not} diffeomorphic, hence neither are $W$ and $W'$. Finally, by inspecting Figure~\ref{fig:dbc}(i), we see that $\Sigma(B^4,D)$ is contractible, hence so are $W$ and $W'$.   \end{proof}

\subsection{Exotic complex curves in open Stein surfaces.} \label{subsec:stein} We now produce exotically knotted, noncompact complex curves that are embedded in open complex surfaces. 

\begin{proof}[Proof of Theorem~\ref{thm:open}]
Consider the exotic holomorphic disks $C,C'$ in the unit ball $B^4 \subset \cc^2$ constructed in the proof of Theorem~\ref{thm:holomorphic}. Note that their boundaries $K,K' \subset S^3$ do not coincide, but they are isotopic as knots. By the proof of Theorem~\ref{thm:tbd}, there is an unknotted loop $\gamma \subset S^3 \setminus K$ that does not bound a smoothly embedded disk in $B^4 \setminus C$, yet its image under the isotopy from $K$ to $K'$ is a loop $\gamma' \subset S^3 \setminus K'$ that \emph{does} bound an embedded disk in $B^4 \setminus C'$. Before proceeding, we wish to consolidate $\gamma$ and $\gamma'$ into a single unknotted loop $\gamma''$. More specifically, we claim there exists an unknot $\gamma''$ in the complement of $K \cup K'$ (carefully constructed as a connected sum of $\gamma$ and $\gamma'$) such that $K \cup \gamma''$ is isotopic to $K \cup \gamma$, $K' \cup \gamma''$ is isotopic to $K' \cup \gamma'$, and in fact $K$ is isotopic to $K'$ in the complement of $\gamma''$. See Figure~\ref{fig:trick} and Lemma~\ref{lem:trick} below. For notational convenience, we  simply denote $\gamma''$ by $\gamma$.

Observe that $C$ and $C'$ are therefore isotopic through homeomorphisms of $B^4$ that fix a neighborhood of $\gamma$. Indeed, the isotopy that carries $K$ to $K'$ and fixes a neighborhood of $\gamma$ extends to an isotopy of $C$ (supported a collar of the boundary) that brings its boundary into alignment with that of $C'$. It follows that $C,C' \subset B^4$ are topologically isotopic (rel boundary) by Theorem~\ref{thm:iso}.

Now choose a Legendrian representative of the loop $\gamma$ lying a small tubular neighborhood of $\gamma$. Attaching a Stein 2-handle along this Legendrian, we obtain a Stein domain $X$. (Note that this manifold has boundary; we will eventually pass to its interior at the end of the proof.) Observe that $X$ is diffeomorphic to the disk bundle over $S^2$ whose Euler number is one less than the Thurston-Bennequin number of  the chosen Legendrian representative of $\gamma$.  Since the complex structure on $X$ coincides with the complex structure on $B^4 \subset \cc^2$ away from the 2-handle, the complex curves $C,C' \subset B^4$ remain holomorphic in $X$; see Remark~\ref{rem:complex} for further discussion. For convenience, we continue to denote these curves by $C,C' \subset X$.

Since $C$ and $C'$ are isotopic through homeomorphisms of $B^4$ supported away from $\gamma$, they remain topologically isotopic in $X$. However, we claim there is no diffeomorphism of $X$ carrying $C$ to $C'$. To see this, consider the branched double covers $\Sigma(X,C)$ and $\Sigma(X,C')$. Since $\gamma$ bounds a disk in the complement of $C'$, $H_2(X)$ is generated by an embedded 2-sphere disjoint from $C'$, and this lifts to a pair of homologically nontrivial, smoothly embedded 2-spheres in $\Sigma(X,C')$. 

We claim there are no such 2-spheres $S$ in  $\Sigma(X,C)$. To simplify the argument, recall from the proof of Theorem~\ref{thm:holomorphic} that there is a diffeomorphism from $\Sigma(B^4,C)$  to the 4-manifold $\Sigma(B^4,D)$ from the proof of Theorem~\ref{thm:tbd}. Thus $\Sigma(B^4,C)$ has the handle diagram shown in Figure~\ref{fig:lifts-open}, and the smooth unknotted loop $\gamma$ has a lift given by the indicated Legendrian knot $\tilde \gamma$. Letting $n$ denote the number of stabilizations, this Legendrian $\tilde \gamma$ has Thurston-Bennequin number $tb=-n$ and rotation number $r=-n$. The attaching curve for the 2-handle has $tb=1$ and $r=1$. Note that there is a second lift of $\gamma$ that is not pictured, but we can fix a Legendrian representative of it nonetheless.

\begin{figure}\center
\def\svgwidth{.55\linewidth}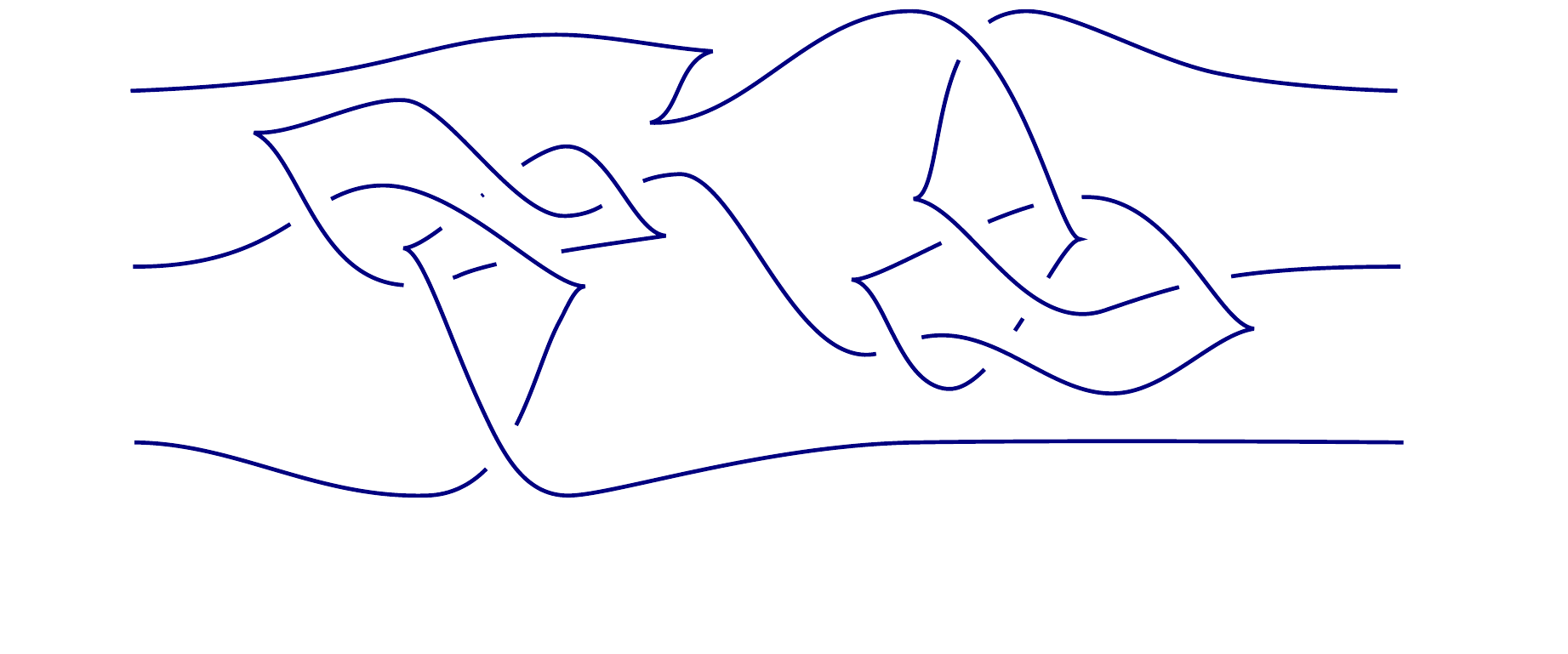
\caption{A stabilized Legendrian $\tilde \gamma$ representing a lift of $\gamma$ to $\Sigma(B^4,C)$.}
\label{fig:lifts-open}
\end{figure}

To obtain $\Sigma(X,C)$, we attach 2-handles along both lifts of $\gamma$. Taking $n$ sufficiently large, $\Sigma(X,C)$ has a Stein handle structure obtained from Figure~\ref{fig:lifts-open} by attaching $(-n-1)$-framed 2-handles along both lifts of $\gamma$. The homology class of the putative 2-sphere $S$ is represented by the difference of the 2-handle from Figure~\ref{fig:lifts-open} and the 2-handle attached along $\tilde \gamma$. Following \cite[Proposition~2.3]{gompf:stein}, the evaluation of $c_1(\Sigma(X,C))$ on $[S]$ is given by the difference in the rotation numbers of the attaching curve of the 2-handle in Figure~\ref{fig:lifts-open} and $\tilde \gamma$, giving
$$|\langle c_1(\Sigma(X,C)), [S] \rangle|=|1-(-n)|=n+1.$$
Since $S$ is homologically nontrivial, it satisfies the adjunction inequality \cite{lisca-matic}:
$$[S] \cdot [S] + | \langle c_1(\Sigma(X,C)), [S] \rangle | \leq 2g(S)-2=-2.$$
The left side of the equation is given by $$[S] \cdot [S]+ \langle c_1(\Sigma(X,C)), [S] \rangle |=(-n-1)+(n+1)=0,$$
yielding a contradiction. It follows that $\Sigma(X,C)$ is not diffeomorphic to $\Sigma(X,C')$.

Finally, consider the open complex curves $\mathring{C}$ and $\mathring{C}'$ in the interior $\mathring{X}$. Note that $\mathring{X}$ is indeed an open Stein surface (cf \cite[\S1.1]{ce:new-applications}). Since $\Sigma(X,C)$ and $\Sigma(X,C')$ are distiguished by their genus functions, their interiors are not diffeomorphic. It follows that the open complex curves $\mathring{C},\mathring{C}'$ in the interior $\mathring{X}$ remain exotic. 
\end{proof}

\begin{rem}\label{rem:complex}
If $X$ is a Stein domain and $X'$ is obtained by attaching a Stein 2-handle along a knot $K \subset \partial X$, then any properly embedded complex curve $C \subset X$ whose boundary is located away from $K$ naturally becomes a complex curve in $X'$. Indeed,  recall that $X'$ can be  obtained by attaching a 2-handle $H$ to $X$ along a nearby complex-analytic Legendrian representative of $K$, then trimming back the boundary slightly; see Figure~\ref{fig:morse} for a schematic and \cite[\S8]{ec:book} for details. By construction, the complex structure on $X'$ coincides with the complex structure on $X$ away from a neighborhood of the 2-handle. Since the complex curve $C \subset X$ is located away from $K$, it remains holomorphic in $X'$. Note that $C$ intersects $X'$ in a slightly shrunken copy of itself, since the boundary of $X\cup H$ has been trimmed back slightly.
\end{rem}

\begin{figure}[h]\center
\def\svgwidth{.55\linewidth}
\begingroup%
  \makeatletter%
  \providecommand\color[2][]{%
    \errmessage{(Inkscape) Color is used for the text in Inkscape, but the package 'color.sty' is not loaded}%
    \renewcommand\color[2][]{}%
  }%
  \providecommand\transparent[1]{%
    \errmessage{(Inkscape) Transparency is used (non-zero) for the text in Inkscape, but the package 'transparent.sty' is not loaded}%
    \renewcommand\transparent[1]{}%
  }%
  \providecommand\rotatebox[2]{#2}%
  \newcommand*\fsize{\dimexpr\f@size pt\relax}%
  \newcommand*\lineheight[1]{\fontsize{\fsize}{#1\fsize}\selectfont}%
  \ifx\svgwidth\undefined%
    \setlength{\unitlength}{921.41304557bp}%
    \ifx\svgscale\undefined%
      \relax%
    \else%
      \setlength{\unitlength}{\unitlength * \real{\svgscale}}%
    \fi%
  \else%
    \setlength{\unitlength}{\svgwidth}%
  \fi%
  \global\let\svgwidth\undefined%
  \global\let\svgscale\undefined%
  \makeatother%
  \begin{picture}(1,0.43386189)%
    \lineheight{1}%
    \setlength\tabcolsep{0pt}%
    \put(0,0){\includegraphics[width=\unitlength,page=1]{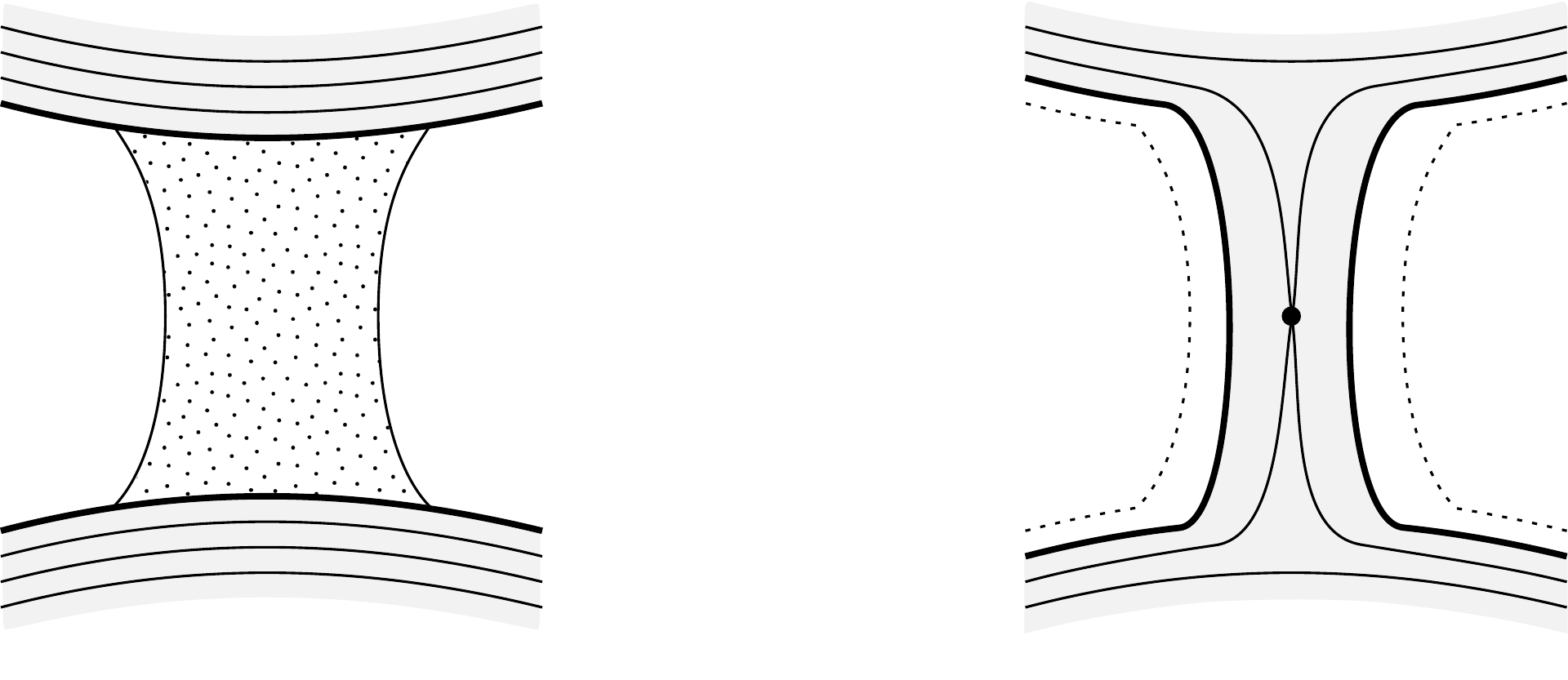}}%
    \put(0.03444554,0.21542473){\color[rgb]{0,0,0}\makebox(0,0)[lt]{\lineheight{1.25}\smash{\begin{tabular}[t]{l}$H$\end{tabular}}}}%
    \put(0.14577316,0.00319017){\color[rgb]{0.4,0.4,0.4}\makebox(0,0)[lt]{\lineheight{1.25}\smash{\begin{tabular}[t]{l}$X$\end{tabular}}}}%
    \put(-0.08436891,0.0806448){\color[rgb]{0,0,0}\makebox(0,0)[lt]{\lineheight{1.25}\smash{\begin{tabular}[t]{l}$\partial X$\end{tabular}}}}%
    \put(0.80451232,0.00319017){\color[rgb]{0.4,0.4,0.4}\makebox(0,0)[lt]{\lineheight{1.25}\smash{\begin{tabular}[t]{l}$X'$\end{tabular}}}}%
    \put(0.55790937,0.06173353){\color[rgb]{0,0,0}\makebox(0,0)[lt]{\lineheight{1.25}\smash{\begin{tabular}[t]{l}$\partial X'$\end{tabular}}}}%
    \put(0,0){\includegraphics[width=\unitlength,page=2]{morse.pdf}}%
  \end{picture}%
\endgroup%

\caption{Trimming back the boundary of $X \cup H$ to obtain $X'$.}
\label{fig:morse}
\end{figure}

\smallskip

Finally, we verify the lemma used in the proof of Theorem~\ref{thm:open}.

\begin{lem}\label{lem:trick}
Let $K$ and $K'$ be isotopic oriented knots in $S^3$, and let $\gamma$ and $\gamma'$ be oriented unknots such that the oriented links $K \cup \gamma$ and $K \cup \gamma'$ are isotopic. Then there exists a third unknot $\gamma''$ such that $K \cup \gamma''$ is isotopic to $K \cup \gamma$, $K' \cup \gamma''$ is isotopic to $K' \cup \gamma'$, and $K$ and $K'$ are isotopic in the complement of $\gamma''$.
\end{lem}

\begin{figure}\center
\def\svgwidth{.9\linewidth}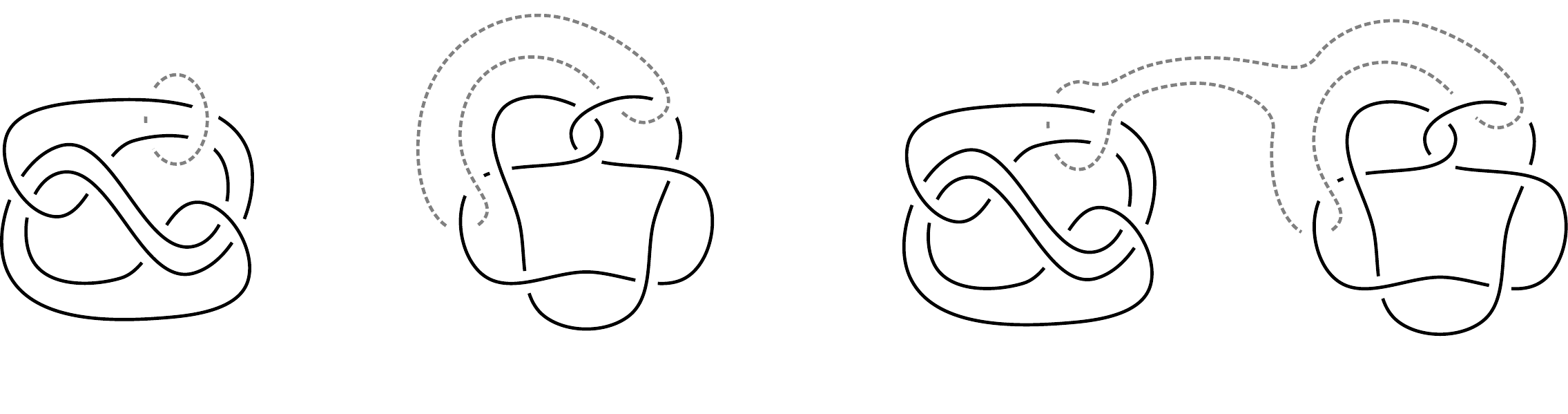
\caption{The links $K \cup \gamma$ and $K' \cup \gamma'$ are isotopic, and in fact $K$ is isotopic to $K'$ in the complement of $\gamma''$, the connected sum of $\gamma$ and $\gamma'$.}
\label{fig:trick}
\end{figure}

\begin{proof}
First, by an isotopy of $\gamma$ in $S^3 \setminus K$ (that possibly passes $\gamma$ through $K' \cup \gamma'$), we may assume that $\gamma$ is split from $K' \cup \gamma'$. We may similarly assume $\gamma'$ is split from $K \cup \gamma$. Now let $\gamma''$ be any unknot formed as a connected sum of  $\gamma$ and $\gamma'$ along an embedded band that is disjoint from $K$ and $K'$. Since $\gamma'$ is an unknot split from $K$, the connected sum $\gamma''$ of $\gamma$ and $\gamma'$ is isotopic to $\gamma$ in $S^3 \setminus K$, hence $K\cup \gamma''$ is isotopic to $K \cup \gamma$. Similarly, $K' \cup \gamma''$ is isotopic to $K' \cup \gamma'$. 

It remains to show that there is an isotopy of $S^3$ carrying $K \cup \gamma''$ to $K' \cup \gamma''$ that fixes $\gamma''$. By the above, $K \cup \gamma''$ is isotopic to $K' \cup \gamma''$. Since the isotopy preserves the orientation on $\gamma''$, we may modify the end of the isotopy so that the induced diffeomorphism of $S^3$ carrying $K \cup \gamma''$ to $K' \cup \gamma''$ restricts to the identity on a neighborhood of $\gamma''$.  Passing to a solid torus $V \cong S^1 \times D^2$ in $S^3$ formed as the exterior of $\gamma''$, we obtain a diffeomorphism of $V$ that is the identity on the boundary and carries $K \subset V$ to $K' \subset V$. Any diffeomorphism of $S^1 \times D^2$ that fixes the boundary is isotopic to the identity (rel boundary), so we conclude that $K$ and $K'$ are isotopic in the solid torus $V$, which lies in the complement of $\gamma''$.\end{proof}

\vspace{-.4in}

\section{Applications to knotted surfaces in closed 4-manifolds}\label{sec:closed}

In this final section, we illustrate further applications of exotically knotted surfaces in the 4-ball to the broader study of knotted surfaces in closed 4-manifolds.

\vspace{-.1in}

\subsection{Knotted tori in 4-manifolds with arbitrary fundamental group.}\label{subsec:tori}

Our goal in this subsection is to show that the surfaces constructed in previous sections can be used to construct  exotically knotted surfaces of positive genus in \emph{closed} 4-manifolds (including those which are not simply connected).

\begin{rem}
To construct topological isotopies between exotically knotted surfaces in a 4-manifold $Z$, most constructions in the literature require  $Z$ to be simply connected (or that $\pi_1(Z)$ is at least ``good'' in the sense of surgery theory). However, such examples can sometimes be promoted to lie in 4-manifolds of the form $Z \# Z'$, where the fundamental group of $Z'$ is unconstrained. In particular, by modifying the topological isotopy between the surfaces to fix a small 4-ball in $Z$ where the connected sum operation is performed, the surfaces remain topologically isotopic in $Z \# Z'$. Under certain conditions on the surfaces and the 4-manifolds $Z$ and $Z'$, the induced surfaces in $Z \# Z'$ can still be distinguished using Bauer-Furuta invariants (which are more robust under connected sums than Seiberg-Witten invariants).
\end{rem}

We produce exotically knotted surfaces in non-simply-connected 4-manifolds without appealing to the connected sum construction sketched above.

\begin{thm}\label{thm:fun-group}
For every finitely presented group $G$, there is a closed 4-manifold $Z$ with $\pi_1(Z)  \cong G$ containing a pair of exotically knotted tori. Moreover, we may assume that there are no essential 3-spheres that are disjoint from both tori. 
\end{thm}

In fact, it is straightforward to arrange for the ambient 4-manifold $Z$ in Theorem~\ref{thm:fun-group} to be irreducible. However, to simplify the argument, we content ourselves with ensuring that just one of the tori has irreducible complement.

\begin{proof}
Let $K\subset S^3$ be the knot represented by the dashed curve in Figure~\ref{fig:torus-trace}(a), where $S^3$ is depicted as zero-surgery on the Hopf link. This knot coincides with $K_0$ in Figure~\ref{fig:main-link} from the proof of Theorem~\ref{thm:tbd}. Applying the construction from \S\ref{sec:simple}, $K$ bounds a pair of disks  $D,D' \subset B^4$ that are topologically isotopic (and indeed coincide with those considered in Theorem~\ref{thm:tbd}). Let $X$ be the 4-manifold obtained from $B^4$ by attaching a zero-framed ``genus-one handle'' along $K$ as defined in \cite[\S3]{hp:embedding}. In particular, letting $F$ denote a compact surface of genus one with one boundary component, we glue $F \times D^2$ to $B^4$ by identifying $\partial F \times D^2$ with a tubular neighborhood of $K$. A handle diagram of $X$ is shown in Figure~\ref{fig:torus-trace}(b). Let $T,T' \subset X$ denote the tori obtained by gluing the slice disk $D,D' \subset B^4$ to the core $F \times 0$ of the genus-one handle.

\begin{figure}\center
\smallskip
\def\svgwidth{\linewidth}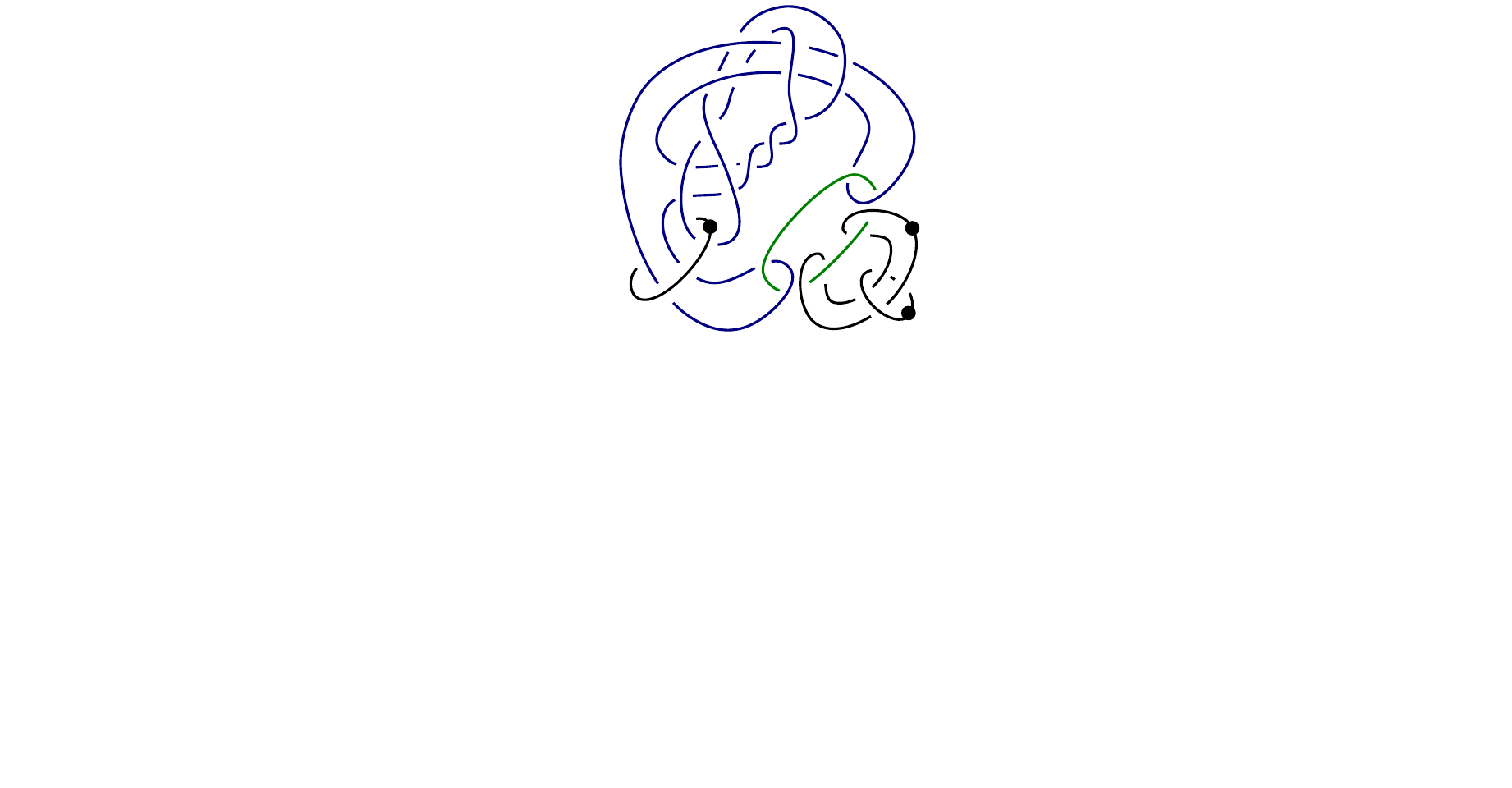
\caption{Part (a) depicts a slice knot $K$ in $S^3$. Parts (b) and (c) depict the 4-manifolds $X$ and $X_T$ in the proof of Theorem~\ref{thm:fun-group}. Part (d) embeds $X_T$ in a larger Stein domain.}\label{fig:torus-trace}
\end{figure}

We construct the desired 4-manifold $Z$ in a few steps. First, let $X_T$ denote the 4-manifold obtained by performing a logarithmic transform with multiplicity zero on $T\subset X$; see \cite{GompfStipsicz4} for background. Following \cite[Figure~8.25]{GompfStipsicz4}, a handle diagram for $X_T$ is given in  Figure~\ref{fig:torus-trace}(c).

Next, we will embed $X_T$ into a closed, minimal symplectic 4-manifold with the desired fundamental group $G$; here a symplectic 4-manifold is \emph{minimal} if it contains no smoothly embedded 2-spheres of self-intersection $-1$. To do so, we first attach four 2-handles to $X_T$ to embed it in a simply connected Stein domain, as illustrated in Figure~\ref{fig:torus-trace}(d).  By \cite[Theorem~2.1]{a-o:topology}, this Stein domain embeds into a simply connected, closed, minimal symplectic 4-manifold $Q$. Moreover, it is straightforward to arrange that $b_2^+(Q)>1$ and that the symplectic cap $C=Q \setminus \mathring{X}_T$ contains a homologically essential Lagrangian torus of square zero with simply connected complement; see, for example,  \cite[Theorem~1.13]{emm:exotic}. We now modify the fundamental group of the cap: By  \cite[Theorem~4.1]{gompf:construction}, every finitely presented group $G$ arises as the fundamental group of a closed, minimal symplectic 4-manifold $Y$ \cite[Theorem~4.1]{gompf:construction}. Moreover, Gompf shows that $Y$ may be assumed to contain a Lagrangian torus of square zero for which the inclusion-induced map $\pi_1(T^2) \to \pi_1(Y)$ is zero. Let us modify $C$ by taking its symplectic sum (as in \cite{gompf:construction}) with $Y$ along the distinguished Lagrangian tori in each of these pieces. The modified cap $C'$ now has fundamental group $G$. Let $Z_T$ be the closed, symplectic 4-manifold $X_T \cup C'$, which has $\pi_1(Z_T) \cong G$. Note that, since $Z_T$ is the symplectic sum of the minimal symplectic 4-manifolds $X_T \cup C$ and $Y$ along a surface of positive genus, $Z_T$  is  minimal by \cite{usher:minimality}.

Now observe that $Z=X \cup C'$ is a closed 4-manifold with $\pi_1(Z) \cong G$. Performing a multiplicity zero log transform along $T$ inside $X \subset Z$ yields the minimal symplectic 4-manifold $Z_T$. Note that, since $Z_T$ is a simply connected, minimal symplectic 4-manifold with $b_2^+>1$, it is irreducible by \cite{kotschick:survey}; this implies that $Z \setminus T$ is irreducible.

\begin{figure}[t]\center
\smallskip
\def\svgwidth{.75\linewidth}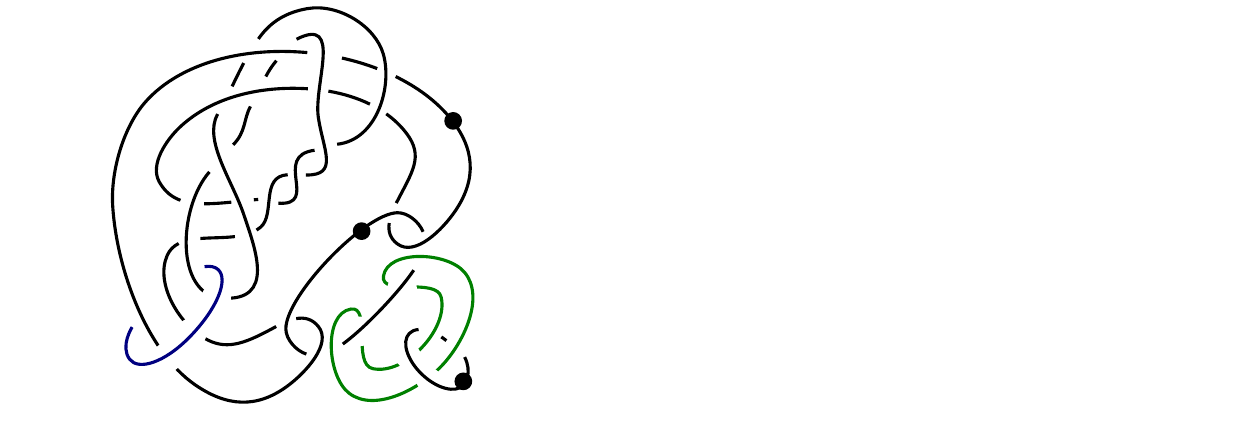
\caption{The 4-manifold $X_{T'}$ and the result of attaching a subset of the additional 2-handles from the cap $Q$ to $X_{T'}$.}\label{fig:torus-trace-prime}
\end{figure}

On the other hand, it is straightforward to see that performing a multiplicity zero log transform on $T'$ inside $X \subset Z$ yields a 4-manifold $Z_{T'}=X_{T'} \cup C'$ that is \emph{not} minimal. To see this, observe that $X_{T'}$ has the handle diagram shown in Figure~\ref{fig:torus-trace-prime}(a), so $Z_{T'}=X_{T'} \cup C'$ contains the 4-manifold shown in Figure~\ref{fig:torus-trace-prime}(b). The latter contains a smooth 2-sphere of square $-1$ formed by the slice disk for the leftmost $-1$-framed unknot in Figure~\ref{fig:torus-trace-prime}(b). It follows that $T$ and $T'$ are not smoothly isotopic in $Z$. \end{proof}

\vspace{-.2in}

\subsection{Exotic surfaces with wild knot groups.}\label{subsec:cuspidal} We can use the same principle of ``localized knotting''  to address a geography problem for knot groups of exotic surfaces. In this subsection, we appeal to Kim's ``twisted'' variant \cite{kim:twist} of  Fintushel-Stern's rim surgery construction \cite{fs:surfaces}.

Let $X$ be a 4-manifold containing a smoothly embedded surface $\surf$, and let $\alpha \subset \surf$ be a nonseparating curve. We may  choose a neighborhood $N \cong S^1 \times B^3$ of $\alpha$ in $X$ such that $\surf \cap N$ is identified with $S^1 \times I$, where $I \subset B^3$ is an unknotted, properly embedded arc;  see Figure~\ref{fig:twist}. Given a knot $K \subset S^3$, we can form a knotted arc $K_+ \subset B^3$ from the connected sum of $I \subset B^3$ and $K \subset S^3$. We may then modify $\surf$ by replacing the pair $S^1\times (B^3,I)$ with $S^1 \times (B^3, K_+)$. We say that the new surface $\surf_K$ in $X$ is obtained from $\surf$ by \emph{rim surgery} along $\alpha$ using the knot $K$.

\begin{figure}[h]\center
\smallskip
\def\svgwidth{.825\linewidth}
\begingroup%
  \makeatletter%
  \providecommand\color[2][]{%
    \errmessage{(Inkscape) Color is used for the text in Inkscape, but the package 'color.sty' is not loaded}%
    \renewcommand\color[2][]{}%
  }%
  \providecommand\transparent[1]{%
    \errmessage{(Inkscape) Transparency is used (non-zero) for the text in Inkscape, but the package 'transparent.sty' is not loaded}%
    \renewcommand\transparent[1]{}%
  }%
  \providecommand\rotatebox[2]{#2}%
  \newcommand*\fsize{\dimexpr\f@size pt\relax}%
  \newcommand*\lineheight[1]{\fontsize{\fsize}{#1\fsize}\selectfont}%
  \ifx\svgwidth\undefined%
    \setlength{\unitlength}{769.71313477bp}%
    \ifx\svgscale\undefined%
      \relax%
    \else%
      \setlength{\unitlength}{\unitlength * \real{\svgscale}}%
    \fi%
  \else%
    \setlength{\unitlength}{\svgwidth}%
  \fi%
  \global\let\svgwidth\undefined%
  \global\let\svgscale\undefined%
  \makeatother%
  \begin{picture}(1,0.2521415)%
    \lineheight{1}%
    \setlength\tabcolsep{0pt}%
    \put(0,0){\includegraphics[width=\unitlength,page=1]{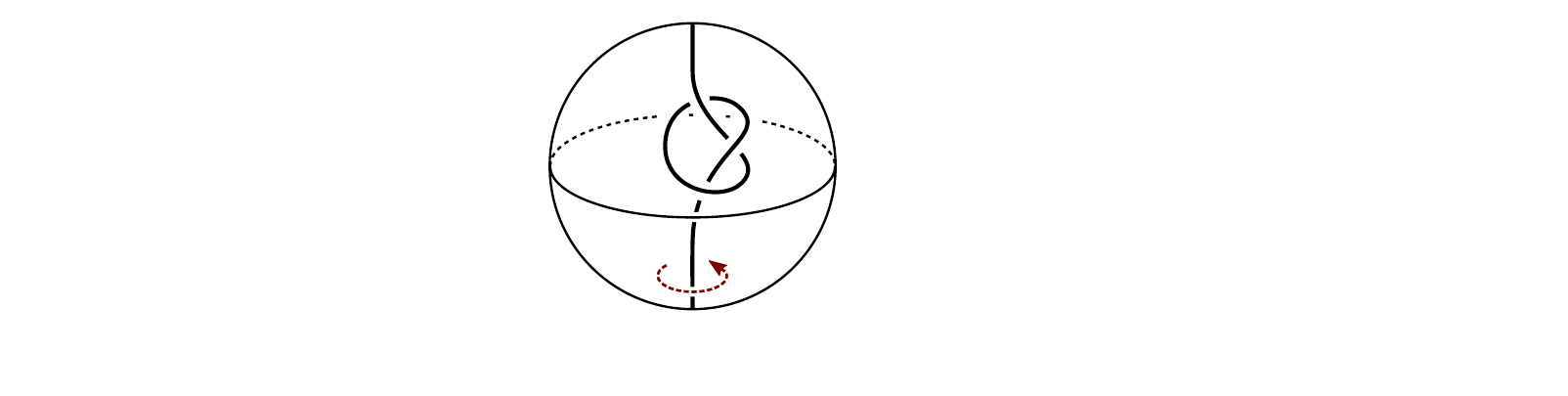}}%
    \put(0.4672212,0.0848561){\color[rgb]{0.50196078,0,0}\makebox(0,0)[lt]{\lineheight{1.25}\smash{\begin{tabular}[t]{l}$\tau$\end{tabular}}}}%
    \put(0.38063171,0.00687018){\color[rgb]{0,0,0}\makebox(0,0)[lt]{\lineheight{1.25}\smash{\begin{tabular}[t]{l}$K_+ \subset B^3$\end{tabular}}}}%
    \put(0.05523659,0.00687018){\color[rgb]{0,0,0}\makebox(0,0)[lt]{\lineheight{1.25}\smash{\begin{tabular}[t]{l}$I \subset B^3$\end{tabular}}}}%
    \put(0,0){\includegraphics[width=\unitlength,page=2]{twist-rim-surgery.pdf}}%
    \put(0.67862006,0.13562558){\color[rgb]{0,0.50196078,0}\makebox(0,0)[lt]{\lineheight{1.25}\smash{\begin{tabular}[t]{l}$\mu$\end{tabular}}}}%
    \put(0.91109863,0.08136454){\color[rgb]{0,0,0.50196078}\makebox(0,0)[lt]{\lineheight{1.25}\smash{\begin{tabular}[t]{l}$\alpha'$\end{tabular}}}}%
    \put(0.68811462,0.00784076){\makebox(0,0)[lt]{\lineheight{1.25}\smash{\begin{tabular}[t]{l}$\mu,\alpha \subset S^1 \! \times_{\tau^d} \! (B^3 \setminus K_+)$\end{tabular}}}}%
  \end{picture}%
\endgroup%

\caption{}\label{fig:twist}
\end{figure}

For Kim's variation on rim surgery, we let $\tau$ be the diffeomorphism of $(B^3,K_+)$ obtained by spinning the arc $K_+$ through a full rotation; see \cite{kim:twist} for a precise definition. The mapping torus of $\tau^d$ defines a twisted annulus $S^1 \times_{\tau^d} (B^3,K_+)$, and we let $\surf_K^{d}$ denote the surface obtained from $\surf$ by replacing $S^1 \times (B^3, I)$ by $S^1 \times_{\tau^d} (B^3,K_+)$; this is known as \emph{$d$-twisted rim surgery}.

The following lemma is inspired by \cite[Lemma~2.1]{kim-ruberman:non-simply}.

\begin{lem}\label{lem:fun-group}
Let $\surf$ be a smooth, closed, oriented surface in a compact 4-manifold $X$ whose meridian is a primitive element of finite order $d > 1$ in $H_1(X \setminus \surf)$. Fix a knot $K$ in $S^3$, and let $G$ denote the fundamental group of the $d$-fold cyclic branched cover of $S^3$ along $K$. Then $\pi_1(X \setminus \surf_K^d)$ has a quotient containing $G$ as a subgroup.
\end{lem}

\begin{proof}
Our notation mostly follows that of Kim \cite{kim:twist}. Let $\mu$ be a meridian to $\surf \subset X$ and let $\alpha'$ be a pushoff of $\alpha$ into $X \setminus \surf$, which we can choose to be nullhomologous in $X \setminus \surf$; see the right side of Figure~\ref{fig:twist} for a schematic.  The group $\pi_1(X \setminus \surf^d_K)$ is obtained by amalgamating $\pi_1(X \setminus \surf)$ and $\pi_1(S^1 \! \times_{\tau^d} \! (B^3 \setminus K_+))$ over $\pi_1(S^1 \! \times \! (\partial B^3 \setminus \partial K_+)) = \langle \mu, \alpha'  \rangle  \cong \zz \oplus \zz.$ 
  In particular,  the Seifert-Van Kampen theorem provides the leftmost square in the following diagram, where the maps $\varphi_i$ and $\psi_i$ are induced by the obvious inclusions (and the remaining maps are explained below):
  \begin{equation}\label{eq:svk2}
\begin{tikzcd}
\pi_1(S^1 \times (\partial B^3 \setminus \partial K_+)) \arrow[r, "\varphi_1", xshift=-4.5mm]  \arrow[d, "\varphi_2", xshift=-4.5mm] &    \pi_1(X \setminus \surf)   \arrow[r, "h", xshift=-14mm, two heads]  \arrow[d, "\psi_1", xshift=-20mm] &   \zz/d \arrow[d, "m", xshift=-23mm]  \\
\pi_1(S^1 \times_{\tau^d} ( B^3 \setminus  K_+)) \arrow[r, "\psi_2", xshift=-27mm] \arrow[rr, bend right=15, "\rho"', xshift=-29.5mm, two heads] &  \pi_1(X \setminus \surf_K^d) \arrow[r, "\exists!", xshift=-37mm, dashed] &  \pi_1(B^3 \setminus K_+)/ \langle \, \mu^d \, \rangle 
\end{tikzcd}
\end{equation}
Following \cite{kim:twist}, the group $\pi_1(S^1\! \times_{\tau^d} \! (B^3 \setminus  K_+)) $ has a simple  presentation:
\begin{align*}
\pi_1(S^1\times_{\tau^d} &(B^3 \setminus  K_+)) \\
&= \langle \, \pi_1(B^3 \setminus K_+), \alpha' \mid    (\alpha')^{-1}\beta \alpha' =  \tau_*^d(\beta) \ \text{for all } \beta \in \pi_1(B^3 \setminus K_+) \, \rangle
\end{align*} 
While $\alpha'$ and $\mu^d$ are nullhomologous in $X \setminus \surf^d_K$, we wish to kill them at the level of $\pi_1$, leading us to a natural quotient of $\pi_1(S^1 \! \times_{\tau^d} \! (B^3 \setminus  K_+))$:
\begin{align*}
&\pi_1(S^1\times_{\tau^d} (B^3 \setminus  K_+)) / \langle \, \alpha', \mu^d \, \rangle 
\\
&=\langle \, \pi_1(B^3 \setminus K_+), \alpha' \mid \,  \alpha'=1, \, \mu^d=1,  \,   (\alpha')^{-1}\beta \alpha' =  \tau_*^d(\beta) \ \text{for all } \beta \in \pi_1(B^3 \setminus K_+) \, \rangle
 \\
 &= \langle \, \pi_1(B^3 \setminus K_+) \mid  \mu^d=1, \,  \beta =  \tau_*^d(\beta) \ \text{for all } \beta \in \pi_1(B^3 \setminus K_+) \, \rangle.
 \end{align*}
To further simplify this, we note as in  \cite[\S2]{kim-ruberman:non-simply} that $\tau$ acts on $\pi_1(B^3 \setminus K_+)$ through conjugation by $\mu$ (viewed as a meridian of $K$), i.e., $\tau_*(\beta)=\mu^{-1} \beta \mu$, so $\tau^d_*(\beta)=(\mu^{-1})^{d} \beta \mu^d$. Therefore, after killing $\mu^d$, we have $\tau_*^d(\beta)=\beta$, hence
$$
\pi_1(S^1\times_{\tau^d} (B^3 \setminus  K_+)) / \langle \, \alpha', \mu^d \, \rangle 
= \pi_1(B^3 \setminus K_+)/ \langle \, \mu^d \, \rangle.$$
Let $\rho: \pi_1(S^1 \! \times_{\tau^d} \! (B^3 \setminus K_+)) \to \pi_1(B^3 \setminus K_+) / \langle \, \mu^d \, \rangle$ denote the quotient map. Next, define $h  :   \pi_1(X \! \setminus \! \surf) \! \to \zz/d$ by composing the Hurewicz map $\pi_1(X \! \setminus \! \surf) \! \to \! H_1(X \! \setminus \! \surf)$ with projection onto the $\zz/d$-summand of $H_1(X \!\setminus \! \surf)$  generated by the meridian to $\surf$. Finally, let $m : \zz/d \to \pi_1(B^3 \setminus K_+)/ \langle \, \mu^d \, \rangle$ be the obvious map sending the generator of $\zz/d$ (which represents the homology class of the meridian $\mu$) back to $\mu$ in $\pi_1(B^3 \setminus K_+)$.

Observe that the outer edges of \eqref{eq:svk2} commute. By the universal property of pushouts, there is a unique map from $\pi_1(X\setminus \surf_K^d)$ to $\pi_1(B^3 \setminus K_+)/\langle \, \mu^d \, \rangle$ making the entire diagram commute. Moreover, this must be a surjection because $\rho$ is a surjection.  Recall that the fundamental group $G$ of the $d$-fold cyclic branched cover of $S^3$ along $K$ is obtained from an index-$d$ subgroup of $\pi_1( B^3 \setminus K_+) \cong \pi_1(S^3 \setminus K)$ after killing the subgroup normally generated by $\mu^d$. It follows that $\pi_1(B^3 \setminus K_+)/\langle \, \mu^d \, \rangle$ is a quotient of  $\pi_1(X \setminus \surf_K^d)$ that contains $G$ as a subgroup. 
\end{proof}

As a final prerequisite for  Theorem~\ref{thm:cuspidal}, we describe a well-behaved class of singular surfaces: As in \cite{golla-starkston}, we say that a \emph{singular symplectic curve} $S$ in a symplectic 4-manifold $X$ is a subset $S \subset X$ that is smooth and symplectic away from a finite set of singularities, where each singular point has an open neighborhood $U$ such that $(U,S \cap U)$ is symplectomorphic to a pair $(V, C \cap V)$ for an open set $V \subset \cc^2$ and a complex curve $C \subset \cc^2$. 

A singular point $x \in S$ is a \emph{cusp} if it is locally modeled on the complex curve  $z^p + w^q=0$ for relatively prime integers $p,q \geq 2$. We may \emph{resolve} a cusp by patching in a smooth surface modeled on the nonsingular complex curve $z^p + w^q =\epsilon$ for sufficiently small $\epsilon>0$.  This patching can be done symplectically; see, for example, \cite[\S2]{etnyre-golla} (especially Lemma 2.2 and the surrounding discussion).

We now restate and prove Theorem~\ref{thm:cuspidal}.

\begin{thmcusp*}
Suppose  $\surf$ is a smooth surface obtained by resolving a cusp of a singular symplectic curve in a closed or convex symplectic 4-manifold $X$.\begin{enumerate}
\item[\normalfont \text{(a)}] If $\surf \cdot \surf \geq 0$, then there are infinitely many smoothly embedded surfaces in $X$ that are topologically but not smoothly isotopic to $\surf$.
\item [\normalfont \text{(b)}] Moreover, if $g(\surf)>1$ and the meridian to $\surf$ is a primitive element of finite order $d>1$ in $H_1(X \setminus \surf)$, then the  class $[\surf]  \in H_2(X)$ contains infinitely many pairwise exotic surfaces whose knot groups contain nonabelian free subgroups.
\end{enumerate}
\end{thmcusp*}

\begin{proof}
The surface $F$ is obtained by resolving a cusp of a singular symplectic curve, so there is an embedded $B^4  \subset X$ such that the subsurface $F_-=F \cap B^4$ is modeled on the smooth curve 
\begin{equation}\label{eq:cusp}
\{(z,w) \in \cc^2: z^p + w^q = \epsilon, \, |z|^2+|w|^2 \leq 1\}
\end{equation}
 for a small value  $\epsilon>0$ where $p,q \geq 2$ are coprime. This surface $\surf_- \subset B^4$ is isotopic to a Seifert surface for the $(p,q)$-torus knot, hence has positive genus and $\pi_1(B^4 \setminus \surf_-)\cong \zz$.

Next, choose an infinite family of knots $\{K_n\}_{n=1}^\infty$ such that the Alexander polynomials of $K_i$ and $K_j$ have distinct sets of coefficients for $i \neq j$; for example, $K_n$ could be chosen to be a fibered knot of genus $n$.  Then choose a simple closed curve $\gamma_- \subset \surf_-$ that is homologically essential in $\surf$. Let $\surf_{i}$ denote the surface obtained from $\surf$ by performing 1-twisted rim surgery with the knot $K_i$ along $\gamma_-$.

We first claim that the surfaces $\surf_i$ and $\surf_{j}$ are  topologically isotopic. These surfaces coincide outside of the chosen 4-ball in $X$, where they are each obtained from the same twisted rim surgery on $\surf_-$. Though $\surf_i$ and $ \surf_{j}$ differ in the interior of this 4-ball, their subsurfaces in the 4-ball are each topologically isotopic rel boundary to $\surf_-$ by \cite[Theorem~1.6]{conway-powell:Z}. It follows that $\surf_i$ and $\surf_{j}$ are topologically isotopic to each other.

Since $\surf$ is symplectic with $\surf \cdot \surf \geq 0$ and the Alexander polynomials of $K_i$ and $K_j$ have distinct sets of coefficients, the pairs $(X,\surf_i)$ and $(X,\surf_j)$ are smoothly distinct for all $i \neq j$ \cite{fs:surfaces,kim:twist}.  This establishes  part (a) of the theorem.

To prove (b), we first narrow the above construction slightly: the surface in \eqref{eq:cusp} can be smoothly isotoped to intersect a 4-ball of smaller radius $\delta<1$ along the surface $$\{(z,w) \in \cc^2: z^2 + w^3 = \epsilon, \, |z|^2+|w|^2 \leq \delta^2\},$$
which is isotopic to the Seifert surface of the right-handed trefoil. Let $\surf_- \subset \surf$ denote this  subsurface, and carry out the construction used to prove part (a), producing a family of exotically knotted surfaces $\surf_i$ that are topologically isotopic to $\surf$.

We now perform a second round of rim surgery: Choose another simple closed curve $\gamma_+$ in $\surf_+=\surf \setminus \mathring{\surf}_-$ that is homologically essential in $\surf$; this is possible because $g(\surf_+)=g(\surf)-1>0$. We also wish to fix a knot $J \subset S^3$  such that the fundamental group of the $d$-fold cyclic branched cover of $S^3$ along $J$ contains a nonabelian free subgroup. To this end, we recall that fundamental groups of closed, hyperbolic 3-manifolds are known to contain nonabelian free subgroups; see \cite{3-manifold-groups}, especially Diagram 1. Thus it suffices to take $J$ to be the knot $8_{18}$ or any other hyperbolic knot whose double branched cover is hyperbolic; the higher-order cyclic branched covers are then guaranteed to be hyperbolic by the orbifold theorem (see \cite{boileau-porti,cooper-hodgson-kerckhoff}),  which shows that the $n$-fold cyclic branched cover of a hyperbolic knot $J$ in $S^3$ is  hyperbolic for $n \geq 3$ unless $n=3$ and $J$ is the figure-eight knot (in which case the double branched cover is also not hyperbolic).

Finally, let $\surf'_i$ be the surface obtained from $\surf_i$ by performing $d$-twisted rim surgery along $\gamma_+$ using the knot $J$. Since $\surf'_i$ and $\surf'_j$ coincide outside the small chosen 4-ball where we performed 1-twisted rim surgery on $\surf_- \subset B^4$, they are again topologically isotopic by \cite{conway-powell:Z}. By Lemma~\ref{lem:fun-group}, each knot group $\pi_1(X \setminus \surf'_i)$ has a quotient containing $G$ as a subgroup. Since $G$ contains a nonabelian free subgroup, it is easy to see  that $\pi_1(X \setminus \surf'_i)$ itself must contain a nonabelian free subgroup.

Finally, to distinguish $\surf'_i$ and $\surf'_j$ smoothly, we recall that the relative Seiberg-Witten invariant of $(X,\surf'_i)$ is obtained from that of $(X,\surf_i)$ by multiplying by the Alexander polynomial   $\Delta_J(t)$ of $J$ \cite{fs:surfaces,kim:twist}, hence from that of $(X,\surf)$ by multiplying by $\Delta_{K_i}(t) \Delta_J(t)$. The coefficients of $\Delta_{K_i}(t) \Delta_J(t)$ and $\Delta_{K_j}(t) \Delta_J(t)$ still differ whenever $i \neq j$, so the pairs $(X,\surf'_i)$ and $(X,\surf'_j)$ with $i \neq j$ are still distinguished up to diffeomorphism by their relative Seiberg-Witten invariants.
\end{proof}

\begin{rem}
This argument, hence Theorem~\ref{thm:cuspidal}, applies equally well to smoothings of reducible singularities modeled on the curve $z^p + w^q =0$ where the integers $p,q \geq 2$ are \emph{not}  coprime, provided we exclude $p=q=2$.
\end{rem}

\subsection{Exotically knotted singular surfaces in 4-space.} As mentioned in the introduction, we can also consider piecewise-linearly embedded surfaces. In particular, we say that a map between smooth manifolds $M$ and $N$ is  \emph{piecewise-linear} if there exist triangulations of $M$ and $N$ compatible with their smooth structures with respect to which the map is simplicial, as defined in \cite[\S2.C]{hatcher}. Using Theorem~\ref{thm:holomorphic}, we obtain exotically knotted PL surfaces in $S^4$. 

\begin{proof}[Proof of Proposition~\ref{prop:PL}]
Let $F,F' \subset B^4$ be exotic surfaces of genus $g$ bounded by the same knot $K \subset S^3$ as constructed in the proof of Theorem~\ref{thm:tbd}. Let $S,S' \subset S^4$ be piecewise-linearly embedded surfaces obtained from $F,F' \subset B^4$ by taking the cone on $K$. Since $F,F' \subset B^4$ are topologically isotopic rel boundary, the surfaces $S,S' \subset S^4$ are topologically isotopic.

For the sake of contradiction, suppose there is a PL homeomorphism $f$ of $S^4$ taking $S$ to $S'$. Choose a small PL regular neighborhood $V$ of the cone point in $S$, where $V$ is PL homeomorphic to $B^4$. The pair $(S^4 \setminus \mathring{V}, S \setminus \mathring{V})$ is PL homeomorphic to $(B^4,F)$. The PL homeomorphism $f: (S^4,S) \to (S^4,S')$ induces a PL homeomorphism
$$(B^4,F) \underset{\textsc{PL}}{\cong} (S^4 \setminus \mathring{V}, S \setminus \mathring{V}) \longrightarrow (S^4 \setminus f(\mathring{V}),S' \setminus f(\mathring{V})) \underset{\textsc{PL}}{\cong} (B^4, F').$$
However, no such PL homeomorphism can exist. Indeed, this would induce a PL homeomorphism of the branched covers $\Sigma(B^4,F)$ and $\Sigma(B^4,F')$. Any PL 4-manifold admits a unique smooth structure up to diffeomorphism; see, e.g., the survey in \cite[\S1]{benedetti}. It follows that $\Sigma(B^4,F)$ and $\Sigma(B^4,F')$ are diffeomorphic, contradicting the proof of Theorem~\ref{thm:tbd}.
\end{proof}

\vspace{-.1in}

\titleformat{\section}
  {\normalfont\fontsize{14}{5}\bfseries}{}{1em}{}
{\small \footnotesize \bibliographystyle{alphamod}
\bibliography{biblio}}

\begin{thebibliography}{AKMR15}

\bibitem[AEMS08]{aems}
{\bfseries A~Akhmedov, J~Etnyre, T~Mark, I~Smith}, {\em A note on {S}tein
  fillings of contact manifolds}, Math. Res. Lett., 15(6):1127--1132, 2008.

\bibitem[AFW15]{3-manifold-groups}
{\bfseries M~Aschenbrenner, S~Friedl, H~Wilton}, \href
  {https://doi.org/10.4171/154} {{\em 3-manifold groups}}, EMS Series of
  Lectures in Mathematics, European Mathematical Society (EMS), Z\"{u}rich,
  2015.

\bibitem[Akb91]{akbulut:zeeman}
{\bfseries S~Akbulut}, \href {https://doi.org/10.1016/0040-9383(91)90028-3}
  {{\em A solution to a conjecture of {Z}eeman}}, Topology, 30(3):513--515,
  1991.

\bibitem[Akb15]{akbulut:2-spheres}
{\bfseries S~Akbulut}, {\em Isotoping 2-spheres in 4-manifolds}, In {\em
  Proceedings of the {G}\"{o}kova {G}eometry-{T}opology {C}onference 2014},
  pages 264--266, G\"{o}kova Geometry/Topology Conference (GGT), G\"{o}kova,
  2015.

\bibitem[AKMR15]{akmr:stable}
{\bfseries D~Auckly, H J Kim, P~Melvin, D~Ruberman}, \href
  {https://doi.org/10.1112/jlms/jdu075} {{\em Stable isotopy in four
  dimensions}}, J. Lond. Math. Soc. (2), 91(2):439--463, 2015.

\bibitem[AM97]{akbulut-matveyev}
{\bfseries S~Akbulut, R~Matveyev}, {\em Exotic structures and adjunction
  inequality}, Turkish J. Math., 21(1):47--53, 1997.

\bibitem[AO02]{a-o:topology}
{\bfseries S~Akbulut, B~Ozbagci}, \href
  {https://doi.org/10.1155/S1073792802108105} {{\em On the topology of compact
  {S}tein surfaces}}, Int. Math. Res. Not., (15):769--782, 2002.

\bibitem[AO14]{ao:singularity}
{\bfseries A~Akhmedov, B~Ozbagci}, \href
  {https://doi.org/10.5427/jsing.2014.8d} {{\em Singularity links with exotic
  {S}tein fillings}}, J. Singul., 8:39--49, 2014.

\bibitem[AY14]{akbulut-yasui:small-stein}
{\bfseries S~Akbulut, K~Yasui}, \href
  {http://projecteuclid.org.ezproxy.cul.columbia.edu/euclid.jsg/1433196060}
  {{\em Infinitely many small exotic {S}tein fillings}}, J. Symplectic Geom.,
  12(4):673--684, 2014.

\bibitem[Ben16]{benedetti}
{\bfseries B~Benedetti}, {\em Smoothing discrete morse theory}, Annali della
  Scuola normale superiore di Pisa-Classe di scienze, 16(2):335--368, 2016.

\bibitem[BH16]{baykur-hayano}
{\bfseries R I Baykur, K~Hayano}, \href
  {https://doi.org/10.2140/gt.2016.20.2335} {{\em Multisections of {L}efschetz
  fibrations and topology of symplectic 4-manifolds}}, Geom. Topol.,
  20(4):2335--2395, 2016.

\bibitem[BO01]{bo:qp}
{\bfseries M~Boileau, S~Orevkov}, {\em Quasi-positivit\'e d'une courbe
  analytique dans une boule pseudo-convexe}, C. R. Acad. Sci. Paris S\'er. I
  Math., 332(9):825--830, 2001.

\bibitem[Bow10]{bowden:thesis}
{\bfseries J~Bowden}, {\em Two-dimensional foliations on four-manifolds}, PhD
  thesis, Ludwig-Maximilians-Universit{\"a}t M{\"u}nchen, 2010.

\bibitem[Boy86]{boyer}
{\bfseries S~Boyer}, \href {https://doi.org/10.2307/2000623} {{\em
  Simply-connected {$4$}-manifolds with a given boundary}}, Trans. Amer. Math.
  Soc., 298(1):331--357, 1986.

\bibitem[BP01]{boileau-porti}
{\bfseries M~Boileau, J~Porti}, {\em Geometrization of 3-orbifolds of cyclic
  type}, Ast\'{e}risque, (272):208, 2001, Appendix A by Michael Heusener and
  Porti.

\bibitem[BST15]{bst:construct}
{\bfseries F~Bourgeois, J M Sabloff, L~Traynor}, \href
  {https://doi-org.proxy.bc.edu/10.2140/agt.2015.15.2439} {{\em Lagrangian
  cobordisms via generating families: construction and geography}}, Algebr.
  Geom. Topol., 15(4):2439--2477, 2015.

\bibitem[CE12]{ec:book}
{\bfseries K~Cieliebak, Y~Eliashberg}, \href {https://doi.org/10.1090/coll/059}
  {{\em From {S}tein to {W}einstein and back}}, volume~59 of {\em American
  Mathematical Society Colloquium Publications}, American Mathematical Society,
  Providence, RI, 2012, Symplectic geometry of affine complex manifolds.

\bibitem[CE20]{ce:new-applications}
{\bfseries K~Cieliebak, Y~Eliashberg}, \href
  {https://doi.org/10.1007/s12220-020-00395-1} {{\em New applications of
  symplectic topology in several complex variables}}, The Journal of Geometric
  Analysis, 2020.

\bibitem[CHK00]{cooper-hodgson-kerckhoff}
{\bfseries D~Cooper, C D Hodgson, S P Kerckhoff}, {\em Three-dimensional
  orbifolds and cone-manifolds}, volume~5 of {\em MSJ Memoirs}, Mathematical
  Society of Japan, Tokyo, 2000, With a postface by Sadayoshi Kojima.

\bibitem[CP19]{conway-powell}
{\bfseries A~Conway, M~Powell}, {\em Enumerating homotopy-ribbon slice discs},
  arXiv preprint arXiv:1902.05321, 2019.

\bibitem[CP20]{conway-powell:Z}
{\bfseries A~Conway, M~Powell}, {\em Embedded surfaces with infinite cyclic
  knot group}, arXiv preprint arXiv:2009.13461, 2020.

\bibitem[DHL15]{dhl}
{\bfseries N M Dunfield, N R Hoffman, J E Licata}, \href
  {https://doi.org/10.4310/MRL.2015.v22.n6.a7} {{\em Asymmetric hyperbolic
  {$L$}-spaces, {H}eegaard genus, and {D}ehn filling}}, Math. Res. Lett.,
  22(6):1679--1698, 2015.

\bibitem[DR16]{rizell:surgery}
{\bfseries G~Dimitroglou~Rizell}, \href
  {https://doi-org.proxy.bc.edu/10.4310/JSG.2016.v14.n3.a6} {{\em Legendrian
  ambient surgery and {L}egendrian contact homology}}, J. Symplectic Geom.,
  14(3):811--901, 2016.

\bibitem[EG20]{etnyre-golla}
{\bfseries J B Etnyre, M~Golla}, {\em Symplectic hats}, arXiv preprint
  arxiv:2001.08978, 2020.

\bibitem[EHK16]{ehk:cobordisms}
{\bfseries T~Ekholm, K~Honda, T~K\'alm\'an}, \href
  {https://doi-org.proxy.bc.edu/10.4171/JEMS/650} {{\em Legendrian knots and
  exact {L}agrangian cobordisms}}, J. Eur. Math. Soc. (JEMS),
  18(11):2627--2689, 2016.

\bibitem[EMM19]{emm:exotic}
{\bfseries J B Etnyre, H~Min, A~Mukherjee}, {\em All 3-manifolds are the
  boundary of exotic 4-manifolds}, arXiv preprint arXiv:1901.07964, 2019.

\bibitem[Etg08]{etgu:survey}
{\bfseries T~Etg\"{u}}, \href {https://doi.org/10.1216/RMJ-2008-38-6-1975}
  {{\em Symplectic and {L}agrangian surfaces in 4-manifolds}}, Rocky Mountain
  J. Math., 38(6):1975--1989, 2008.

\bibitem[FQ14]{freedman-quinn}
{\bfseries M H Freedman, F~Quinn}, {\em Topology of 4-{M}anifolds (PMS-39),
  Volume 39}, Princeton University Press, 2014.

\bibitem[Fre82]{freedman}
{\bfseries M H Freedman}, \href
  {http://projecteuclid.org.proxy.bc.edu/euclid.jdg/1214437136} {{\em The
  topology of four-dimensional manifolds}}, J. Differential Geom.,
  17(3):357--453, 1982.

\bibitem[FS97a]{fs:blowdown}
{\bfseries R~Fintushel, R J Stern}, \href
  {http://projecteuclid.org.ezproxy.cul.columbia.edu/euclid.jdg/1214459932}
  {{\em Rational blowdowns of smooth {$4$}-manifolds}}, J. Differential Geom.,
  46(2):181--235, 1997.

\bibitem[FS97b]{fs:surfaces}
{\bfseries R~Fintushel, R J Stern}, \href
  {https://doi.org/10.4310/MRL.1997.v4.n6.a10} {{\em Surfaces in
  {$4$}-manifolds}}, Math. Res. Lett., 4(6):907--914, 1997.

\bibitem[FS99]{fs:symplectic}
{\bfseries R~Fintushel, R J Stern}, \href
  {http://projecteuclid.org.ezproxy.cul.columbia.edu/euclid.jdg/1214425276}
  {{\em Symplectic surfaces in a fixed homology class}}, J. Differential Geom.,
  52(2):203--222, 1999.

\bibitem[FT95a]{freedman-teichner:i}
{\bfseries M H Freedman, P~Teichner}, \href
  {https://doi.org/10.1007/BF01231454} {{\em {$4$}-manifold topology. {I}.
  {S}ubexponential groups}}, Invent. Math., 122(3):509--529, 1995.

\bibitem[FT95b]{freedman-teichner:ii}
{\bfseries M H Freedman, P~Teichner}, \href
  {https://doi.org/10.1007/BF01231455} {{\em {$4$}-manifold topology. {II}.
  {D}wyer's filtration and surgery kernels}}, Invent. Math., 122(3):531--557,
  1995.

\bibitem[Gab01]{gabai:smale}
{\bfseries D~Gabai}, \href
  {http://projecteuclid.org.ezproxy.cul.columbia.edu/euclid.jdg/1090348284}
  {{\em The {S}male conjecture for hyperbolic 3-manifolds: {${\rm
  Isom}(M^3)\simeq{\rm Diff}(M^3)$}}}, J. Differential Geom., 58(1):113--149,
  2001.

\bibitem[Gay02]{gay:2-handles}
{\bfseries D T Gay}, \href {https://doi.org/10.1090/S0002-9947-01-02890-2}
  {{\em Symplectic 2-handles and transverse links}}, Trans. Amer. Math. Soc.,
  354(3):1027--1047, 2002.

\bibitem[Gom95]{gompf:construction}
{\bfseries R E Gompf}, \href {https://doi.org/10.2307/2118554} {{\em A new
  construction of symplectic manifolds}}, Ann. of Math. (2), 142(3):527--595,
  1995.

\bibitem[Gom98]{gompf:stein}
{\bfseries R~Gompf}, {\em Handlebody construction of {S}tein surfaces}, Ann. of
  Math. (2), 148(2):619--693, 1998.

\bibitem[GS99]{GompfStipsicz4}
{\bfseries R E Gompf, A I Stipsicz}, {\em 4-manifolds and Kirby calculus.},
  number~20 in Graduate Studies in Mathematics, American Mathematical Society,
  1999.

\bibitem[GS19]{golla-starkston}
{\bfseries M~Golla, L~Starkston}, {\em The symplectic isotopy problem for
  rational cuspidal curves}, arXiv preprint arXiv:1907.06787, 2019.

\bibitem[Hat76]{hatcher:large}
{\bfseries A~Hatcher}, \href {https://doi.org/10.1016/0040-9383(76)90027-6}
  {{\em Homeomorphisms of sufficiently large {$P^{2}$}-irreducible
  {$3$}-manifolds}}, Topology, 15(4):343--347, 1976.

\bibitem[Hat02]{hatcher}
{\bfseries A~Hatcher}, {\em Algebraic topology}, Cambridge University Press,
  Cambridge, 2002.

\bibitem[Hay17]{hayden:stein}
{\bfseries K~Hayden}, {\em Quasipositive links and {S}tein surfaces},
  Geom.~Topol. (to appear), arXiv:1703.10150, 2017.

\bibitem[Hay20]{hayden:doc}
{\bfseries K~Hayden}, \href {https://doi.org/10.7910/DVN/JWE5AX} {{\em
  {\normalfont Data and code to accompany this paper}}},
  \url{https://doi.org/10.7910/DVN/JWE5AX}, March 2020.

\bibitem[HMP18]{hmp:mazur}
{\bfseries K~Hayden, T E Mark, L~Piccirillo}, {\em Exotic {M}azur manifolds and
  knot trace invariants}, arXiv preprint arXiv:1908.05269, 2018.

\bibitem[HP19]{hp:embedding}
{\bfseries K~Hayden, L~Piccirillo}, {\em The trace embedding lemma and
  spinelessness}, J.~Differential Geom.~(to appear), arXiv:1912.13021, 2019.

\bibitem[HS20]{hoffman-sunukjian}
{\bfseries N R Hoffman, N S Sunukjian}, \href
  {https://doi.org/10.2140/agt.2020.20.2677} {{\em Null-homologous exotic
  surfaces in 4-manifolds}}, Algebr. Geom. Topol., 20(5):2677--2685, 2020.

\bibitem[Iva76]{ivanov}
{\bfseries N~Ivanov}, {\em Research in topology ii}, Notes of LOMI scientific
  seminars, (66):172--176, 1976.

\bibitem[JMZ20]{jmz:exotic}
{\bfseries A~Juh{\'a}sz, M~Miller, I~Zemke}, {\em Transverse invariants and
  exotic surfaces in the 4-ball}, Geom.~Topol. (to appear), arXiv:2001.07191,
  2020.

\bibitem[Joh79]{johannson}
{\bfseries K~Johannson}, {\em Homotopy equivalences of {$3$}-manifolds with
  boundaries}, volume 761 of {\em Lecture Notes in Mathematics}, Springer,
  Berlin, 1979.

\bibitem[JS78]{jaco-shalen}
{\bfseries W~Jaco, P B Shalen}, {\em A new decomposition theorem for
  irreducible sufficiently-large {$3$}-manifolds}, In {\em Algebraic and
  geometric topology ({P}roc. {S}ympos. {P}ure {M}ath., {S}tanford {U}niv.,
  {S}tanford, {C}alif., 1976), {P}art 2}, Proc. Sympos. Pure Math., XXXII,
  pages 71--84, Amer. Math. Soc., Providence, R.I., 1978.

\bibitem[Kim06]{kim:twist}
{\bfseries H J Kim}, \href {https://doi.org/10.2140/gt.2006.10.27} {{\em
  Modifying surfaces in 4-manifolds by twist spinning}}, Geom. Topol.,
  10:27--56, 2006.

\bibitem[Kot97]{kotschick:survey}
{\bfseries D~Kotschick}, {\em The {S}eiberg-{W}itten invariants of symplectic
  four-manifolds (after {C}. {H}. {T}aubes)}, number 241, pages Exp. No. 812,
  4, 195--220, 1997, S\'{e}minaire Bourbaki, Vol. 1995/96.

\bibitem[KR08a]{kim-ruberman:non-simply}
{\bfseries H J Kim, D~Ruberman}, \href
  {https://doi.org/10.2140/agt.2008.8.2263} {{\em Smooth surfaces with
  non-simply-connected complements}}, Algebr. Geom. Topol., 8(4):2263--2287,
  2008.

\bibitem[KR08b]{kim-ruberman:trivial}
{\bfseries H J Kim, D~Ruberman}, \href
  {https://doi.org/10.1090/S0002-9947-08-04482-6} {{\em Topological triviality
  of smoothly knotted surfaces in 4-manifolds}}, Trans. Amer. Math. Soc.,
  360(11):5869--5881, 2008.

\bibitem[Li20]{li:email}
{\bfseries T J Li}, Private communication, March 2020.

\bibitem[Liv04]{livingston:tau}
{\bfseries C~Livingston}, \href {https://doi.org/10.2140/gt.2004.8.735} {{\em
  Computations of the {O}zsv\'{a}th-{S}zab\'{o} knot concordance invariant}},
  Geom. Topol., 8:735--742, 2004.

\bibitem[LM98]{lisca-matic}
{\bfseries P~Lisca, G~Mati{\'c}}, {\em Stein $4$-manifolds with boundary and
  contact structures}, Topology Appl., 88:55--66, 1998.

\bibitem[LW12]{li-wu}
{\bfseries T J Li, W~Wu}, \href {https://doi.org/10.2140/gt.2012.16.1121} {{\em
  Lagrangian spheres, symplectic surfaces and the symplectic mapping class
  group}}, Geom. Topol., 16(2):1121--1169, 2012.

\bibitem[OO99]{oo:simple}
{\bfseries H~Ohta, K~Ono}, \href {https://doi.org/10.1007/s000140050106} {{\em
  Simple singularities and topology of symplectically filling {$4$}-manifold}},
  Comment. Math. Helv., 74(4):575--590, 1999.

\bibitem[OS03]{oz-sz:genus}
{\bfseries P~Ozsv\'{a}th, Z~Szab\'{o}}, \href
  {https://doi.org/10.2140/gt.2003.7.615} {{\em Knot {F}loer homology and the
  four-ball genus}}, Geom. Topol., 7:615--639, 2003.

\bibitem[OS04]{os:surgery}
{\bfseries B~Ozbagci, A I Stipsicz}, \href
  {https://doi.org/10.1007/978-3-662-10167-4} {{\em Surgery on contact
  3-manifolds and {S}tein surfaces}}, volume~13 of {\em Bolyai Society
  Mathematical Studies}, Springer-Verlag, Berlin; J\'anos Bolyai Mathematical
  Society, Budapest, 2004.

\bibitem[PPV07]{park-poddar-vidussi}
{\bfseries B D Park, M~Poddar, S~Vidussi}, \href
  {https://doi.org/10.1090/S0002-9947-07-04168-2} {{\em Homologous non-isotopic
  symplectic surfaces of higher genus}}, Trans. Amer. Math. Soc.,
  359(6):2651--2662, 2007.

\bibitem[Rub98]{ruberman:obstruction}
{\bfseries D~Ruberman}, \href {https://doi.org/10.4310/MRL.1998.v5.n6.a5} {{\em
  An obstruction to smooth isotopy in dimension {$4$}}}, Math. Res. Lett.,
  5(6):743--758, 1998.

\bibitem[Rud83a]{rudolph:braided-surface}
{\bfseries L~Rudolph}, {\em Braided surfaces and {S}eifert ribbons for closed
  braids}, Comment. Math. Helv., 58(1):1--37, 1983.

\bibitem[Rud83b]{rudolph:qp-alg}
{\bfseries L~Rudolph}, {\em Algebraic functions and closed braids}, Topology,
  22(2):191--202, 1983.

\bibitem[Sage]{sagemath}
{\bfseries {The Sage Developers}}, {\em {S}agemath, the {S}age mathematics
  software system}, Available at \url{https://www.sagemath.org}, 2019.

\bibitem[Sch19]{schwartz}
{\bfseries H R Schwartz}, {\em Equivalent non-isotopic spheres in 4-manifolds},
  Journal of Topology, 12(4):1396--1412, 2019.

\bibitem[Snap]{snappy}
{\bfseries M~Culler, N M Dunfield, M~Goerner, J R Weeks}, {\em Snap{P}y, a
  computer program for studying the geometry and topology of $3$-manifolds},
  \url{http://snappy.computop.org}.

\bibitem[ST05]{siebert-tian}
{\bfseries B~Siebert, G~Tian}, \href
  {https://doi.org/10.4007/annals.2005.161.959} {{\em On the holomorphicity of
  genus two {L}efschetz fibrations}}, Ann. of Math. (2), 161(2):959--1020,
  2005.

\bibitem[Sun15]{sunukjian}
{\bfseries N S Sunukjian}, \href {https://doi.org/10.1093/imrn/rnu187} {{\em
  Surfaces in 4-manifolds: concordance, isotopy, and surgery}}, Int. Math. Res.
  Not. IMRN, (17):7950--7978, 2015.

\bibitem[Swe19]{klo}
{\bfseries F~Swenton}, {\em Kirby calculator}, Available at
  \url{https://community.middlebury.edu/~mathanimations/klo/}, 2019.

\bibitem[Sza19]{hfk-calc}
{\bfseries Z~Szab{\'o}}, {\em {K}not {F}loer homology calculator}, Available at
  \url{https://web.math.princeton.edu/~szabo/HFKcalc.html}, 2019.

\bibitem[Tor20]{torres}
{\bfseries R~Torres}, {\em Knotted nullhomologous surfaces in 4-manifolds},
  arXiv preprint arXiv:2005.11696, 2020.

\bibitem[Ush06]{usher:minimality}
{\bfseries M~Usher}, \href {https://doi.org/10.1155/IMRN/2006/49857} {{\em
  Minimality and symplectic sums}}, Int. Math. Res. Not., pages Art. ID 49857,
  17, 2006.

\bibitem[Wen18]{wendl:book}
{\bfseries C~Wendl}, \href {https://doi.org/10.1007/978-3-319-91371-1} {{\em
  Holomorphic curves in low dimensions}}, volume 2216 of {\em Lecture Notes in
  Mathematics}, Springer, Cham, 2018.

\end{thebibliography}

\end{document}